\documentclass[12pt, reqno]{amsart}
\usepackage{amsmath,amstext, amsthm, amscd, amsfonts, amssymb, graphicx, color}
\usepackage[parfill]{parskip}
\usepackage[bookmarksnumbered, colorlinks, plainpages]{hyperref}
\usepackage{bbm}
\usepackage{amssymb}
\usepackage{mathtools}
\usepackage{xcolor}
\usepackage{thmtools,thm-restate}
\usepackage[T1]{fontenc}
\mathtoolsset{showonlyrefs}
\usepackage[]{amsaddr}
\usepackage[square,numbers,sort]{natbib}
\usepackage{etoolbox}

\renewenvironment{titlepage}{%
  \thispagestyle{empty}\setcounter{page}{1}
  \vspace*{\fill}
}{%
  \vspace{1\baselineskip}
  \vspace*{\fill}
}
\makeatletter
\patchcmd{\@maketitle}
{\if@titlepage %\newpage 
\else}
{\if@titlepage
 \else}
{}{}
\makeatother

\usepackage[notocbasic]{nomencl}
\providetoggle{nomsort}
\settoggle{nomsort}{true} 
\usepackage{enumitem}
\usepackage{etoolbox}
\renewcommand\nomgroup[1]{%
  \item[\bfseries
  \ifstrequal{#1}{M}{Model}{%
  \ifstrequal{#1}{N}{Numerical Scheme}{%
  \ifstrequal{#1}{P}{Bayesian inverse problem}{%
  \ifstrequal{#1}{O}{Finite range approximation}{%
  \ifstrequal{#1}{Z}{Other symbols}{}}}}}%
]}

\makenomenclature

\newtheorem{theorem}{Theorem}[section]
\newtheorem{lemma}[theorem]{Lemma}

\newtheorem{corollary}[theorem]{Corollary}
\theoremstyle{definition}
\newtheorem{definition}[theorem]{Definition}

\theoremstyle{remark}
\newtheorem{remark}[theorem]{Remark}
\numberwithin{equation}{section}
\newtheorem{assumptions}[theorem]{Assumptions}
\numberwithin{equation}{section}

\begin{document}
\setcounter{page}{1}

\title[Splitting algorithm on metric spaces]{Operator splitting algorithm for structured population models on metric spaces}

\author[]{Carolin Lindow$^1$, Christian D\"{u}ll$^1$, Piotr Gwiazda$^{2,5}$, Błażej Miasojedow$^{3}$, Anna Marciniak-Czochra$^{1,4,*}$}

\address{\small $^{1}$ Institute for Mathematics, Heidelberg University, Germany
}
\address{\small $^{2}$ Institute of Mathematics of Polish Academy of Sciences, Warsaw, Poland
}
\address{\small$^{3}$ 
Institute of Applied Mathematics and Mechanics, University of Warsaw, Poland
}
\address{\small $^{4}$ Interdisciplinary Center for Scientific Computing (IWR), Heidelberg University, Germany
}
\address{\small $^{5}$ Interdisciplinary Centre for Mathematical and Computational Modelling (ICM), University of Warsaw, Poland
}

 \email{*corresponding author anna.marciniak@iwr.uni-heidelberg.de}

% \email{\textcolor[rgb]{0.00,0.00,0.84}{second@gjom.org}}



 %\newline \indent $^{*}$ Corresponding author

%\begin{abstract}
%In this paper, we propose a numerical scheme for structured population models defined on a separable and complete metric space. In particular, we consider a generalized version of a transport equation with additional growth and non-local interaction terms in the space of nonnegative Radon measures equipped with the flat metric.\\
%The solutions, given by families of Radon measures, are approximated by linear combinations of Dirac measures. For this purpose, we introduce a finite-range approximation of the measure-valued model functions, provided that they are linear. By applying an operator splitting technique, we are able to separate the effects of the transport from those of growth and the non-local interaction. We derive the order of convergence of the numerical scheme, which is linear in the spatial discretization parameters and polynomial of order $\alpha$ in the time step size, assuming that the model functions are $\alpha$ Hölder regular in time.\\
%In a second step, we show that our proposed algorithm can approximate the posterior measure of Bayesian inverse models, which will allow us to link model parameters to measured data in the future.
%\end{abstract} \maketitle

\begin{titlepage}
{\centering
  {\Large Operator splitting algorithm for structured population models on metric spaces\par}
  \vspace{2\baselineskip}
  Carolin Lindow$^1$, Christian D\"{u}ll$^1$, Piotr Gwiazda$^{2,5}$, Błażej Miasojedow$^{3}$, Anna Marciniak-Czochra$^{1,4,*}$\par
  \vspace{1\baselineskip}
  %\today
  }
  {\bf Abstract. } In this paper, we propose a numerical scheme for structured population models defined on a separable and complete metric space. In particular, we consider a generalized version of a transport equation with additional growth and non-local interaction terms in the space of nonnegative Radon measures equipped with the flat metric.\\
The solutions, given by families of Radon measures, are approximated by linear combinations of Dirac measures. For this purpose, we introduce a finite-range approximation of the measure-valued model functions, provided that they are linear. By applying an operator splitting technique, we are able to separate the effects of the transport from those of growth and the non-local interaction. We derive the order of convergence of the numerical scheme, which is linear in the spatial discretization parameters and polynomial of order $\alpha$ in the time step size, assuming that the model functions are $\alpha$ Hölder regular in time.\\
In a second step, we show that our proposed algorithm can approximate the posterior measure of Bayesian inverse models, which will allow us to link model parameters to measured data in the future.\\
{\tiny $^{1}$ Institute for Mathematics, Heidelberg University, Germany\\
$^{2}$ Institute of Mathematics of Polish Academy of Sciences, Warsaw, Poland\\
$^{3}$ 
Institute of Applied Mathematics and Mechanics, University of Warsaw, Poland\\
$^{4}$ Interdisciplinary Center for Scientific Computing (IWR), Heidelberg University, Germany\\
$^{5}$ Interdisciplinary Centre for Mathematical and Computational Modelling (ICM), University of Warsaw, Poland\\
* Corresponding author, email: anna.marciniak@iwr.uni-heidelberg.de
}

{\tiny
{\bf Acknowledgement.} 
This work is funded by the Deutsche Forschungsgemeinschaft (DFG, German Research Foundation) under Germany 's Excellence Strategy EXC 2181/1 - 390900948 (the Heidelberg STRUCTURES Excellence Cluster). \\
Carolin Lindow and Anna Marciniak-Czochra are supported by the European Research Council (ERC) under the European Union’s Horizon 2020 research and innovation programme (project PEPS, no. 101071786).\\
Carolin Lindow is pleased to express gratitude to the Thematic Research Programme funded by the Excellence Initiative Research University at the University of Warsaw for financing her research stay in Warsaw.\\
The work of P. G. was partially supported by the National Science Centre (Poland), agreement no. 2023/51/B/ST1/01546.
{\bf Statements and Declarations.} 
All authors declare that they have no conflicts of interest.
}
\end{titlepage}

%\maketitle
\section{Introduction}

Modeling natural phenomena such as population dynamics with mathematics is a powerful tool in various fields, including epidemiology \cite{Alahmadi2020}, \cite{Brauer2019}, \cite{Padmanabhan2021}, ecology \cite{Tredennick2021}, \cite{vandenBerg2022}, \cite{White2020}, and cell biology \cite{Danciu2023}, \cite{Gwiazda2012}, \cite{Lorenzi2019}. However, the practical value of these models largely depends on the ability to combine the theoretical model with real data \cite{Alahmadi2020}, so that the best parametrization can be identified. This in turn enables a thorough analysis of the underlying mechanisms as well as predictions about the process dynamics and their response to perturbations. It is therefore essential to develop powerful numerical tools for parameterizing  the proposed models. \\

In this paper, we  focus on structured population models formulated in measures, which naturally generalize the classical PDE formulation in an $L^1$ setting as they permit combining discrete and continuous transitions in a single model \cite{Diekmann2019}, \cite{dull2024}, \cite{ düll_gwiazda_marciniak-czochra_skrzeczkowski_2021},  \cite{Gwiazda2012}. It also allows to use initial values with discontinuous but integrable densities or even Radon measures without densities,  and to consider models whose solutions exhibit concentration phenomena as studied in \cite{Busse}. 
Furthermore, experimental measurements often only capture the number of individuals in certain cohorts which we can naturally represent using Dirac measures \cite{Carillo}, \cite{Colombo2018}.\\
In the framework for structured population models developed in \cite{düll_gwiazda_marciniak-czochra_skrzeczkowski_2021}, the following partial differential equation is considered on the space of Radon measures:

\begin{align}
\label{Eq:MeasurePDE}
    \partial_t \mu(t)+\nabla_x\big(b(t,x,\mu(t))\mu(t)\big)&=c(t,x,\mu(t))\\&\quad + \int\limits_{\mathbb R ^d} \eta(t,x,\mu(t))(y)\mathrm{d}[\mu(t)](y)+N(t,\mu(t)).
\end{align}
We shortly explain the model and the notion of solutions. The function $b$ denotes the velocity and direction of the change of the state variable, $c$ represents growth, the measure $\eta$ describes non-local transitions in the state space, and the measure $N$ is an influx function.
The solution $\mu(\cdot)$ takes values in the space of nonnegative Radon measures. Hence, the derivatives of the measure $\mu(\cdot)$ in \eqref{Eq:MeasurePDE} are to be understood in an appropriate weak sense using bounded Lipschitz functions as test functions \cite[Chapter 3.4]{düll_gwiazda_marciniak-czochra_skrzeczkowski_2021}.\\
In this paper, we work with a generalization of \eqref{Eq:MeasurePDE} to separable and complete metric spaces. This allows us to consider models that are based on  graphs, manifolds, or spaces with a non-Euclidean distance function. However, the trade-off for this general setting is that we may not rely on a linear structure, and therefore cannot formulate a PDE like \eqref{Eq:MeasurePDE}. Instead, we introduce an implicit representation formula involving push-forward measures \eqref{eq:mod_pushforward} and derive an adjusted weak formulation \eqref{eq:Weakform}. In Section \ref{Sec:ProblemSetting}, we give an overview of the considered setting, while Appendix \ref{Sec:SolutionsandProperties} summarizes the most important properties of the model and its solutions obtained in \cite[Chapter 3]{düll_gwiazda_marciniak-czochra_skrzeczkowski_2021}.\\

Based on this analytical setting, we propose a numerical approximation to the solutions of structured population models on the space of Radon measures. 
Inspired by particle-based algorithms, the initial measure is approximated by a linear combination of Dirac measures, which is justified by the density of the linear span of Dirac measures under the flat norm \cite[Theorem 1.37]{düll_gwiazda_marciniak-czochra_skrzeczkowski_2021}. The positions and masses of these Diracs are tracked, with only the masses and then only the locations being updated alternately. We provide a discrete approximation for the measure valued model functions in the case that they are linear. This discrete form ensures that the number of Dirac measures in the approximation remains finite.\\ 
 The key idea of our approach is based on semigroup splitting as introduced in \cite{MR2050900} which allows us to consider the growth and transport processes in \eqref{Eq:MeasurePDE} separately instead of simultaneously, which is a common approach in numerical algorithms for differential equations \cite{Blanes2024}. 
 Our proposed algorithm generalizes the particle-based operator splitting approximation introduced in \cite{Carillo} for \eqref{Eq:MeasurePDE} by extending it to the broad class of structured population models on Polish metric spaces, i.e. state spaces which are separable and complete.\\

 The numerical method that we consider here, is based on particle based simulations. Particle methods are used to simulate models with a large number of interacting particles or individuals \cite{Carillo} which appear in various scientific fields \cite{Holm2019} such as pedestrian flow \cite{Etikyala2014}, \cite{Mahato2018}, advection-selection-mutation models \cite{Guilberteau2024}, astrophysics and fluid dynamics \cite{Chertock2017}. A common assumption when applying particle methods is that the mass is conserved. This is however not the case with our structured population models, so that we modify the method to allow for new particles and particles with different masses. This also implies, that particles can no longer be interpreted as individual agents. The weighted sum of particles is an abstract approximation of the density of individuals.\\

Previously, structured population models have been simulated using a variety of numerical methods. As we consider models on Polish metric spaces, we need to adjust the already existing methods to our setting.\\
In a classical Euclidean setting, structured population models are usually described by hyperbolic partial differential equations.
The corresponding numerical methods are thus often adapted from fluid mechanics and include among others explicit and implicit Euler schemes for a fixed discretization of the state space \cite{Sulsky1994}.\\
Other numerical methods for solving structured population models  implement the method of characteristics commonly used to solve transport type PDEs \cite{Sulsky1994}. Recently, spectral methods for structured population models have been suggested \cite{XiaPHD}. Another approach is to adapt finite difference schemes to treat models with growth rates that change in sign, see \cite{Ackleh2019}. In \cite{Joshi2023}, the authors provide a tool to simulate structured population models using a variety of numerical methods.\\
For structured population models on the space of Radon measures, the Escalator Boxcar Train (EBT) algorithm proposed by \cite{deRoos} has been considered on the state space $\mathbb{R}^d$. This method explicitly deals with a non-local boundary term that arises through a birth process. General convergence of the EBT algorithm has been proven in \cite{Brnnstrm2013}, \cite{Gwiazda2014}, \cite{UlikowskaPHD} for Euclidean state spaces, whereas \cite{Gwiazda2022} focuses on the convergence of the method applied to a particular structured population model with  a non-local interaction term in two and three dimensional Euclidean state spaces.

In order to obtain meaningful results in the simulations, we need to parametrize the model. 
One common problem in mathematical modeling is to determine these parameters based on observations of the solution.  
Even if we can solve the forward problem, 
it is possible that the inverse problem 
is not well-posed as the inverse operator fails to be injective or continuous 
\cite{Stuart_2010}.
In this paper, we assume that the model is determined by finitely many  real valued parameters contained within a compact set. We prove that the approximate posterior measure obtained using the solution generated by our algorithm is close to the exact posterior measure in the flat metric.\\

The paper is structured as follows. In Section \ref{Sec:ProblemSetting}, we introduce the model and summarize the most important properties of the space of Radon measures equipped with the flat norm. The numerical scheme is introduced in Section \ref{Sec:NumericalScheme}. Next, in Section \ref{Bayes} we prove the continuity of the posterior density w.r.t. the model functions and the data in a Bayesian setting. The remainder of the paper is dedicated to the error analysis. Section \ref{Sec:approximation} treats the errors arising from the approximation of the initial value and the model functions. Lastly, in Section \ref{Sec:convergence}, we prove the convergence order of the numerical scheme by first showing that the semigroup splitting technique is applicable and then employing it to show the convergence of our numerical scheme. The concept of the proof is illustrated in Figure \ref{fig:overview}. For the reader's convenience, a summary of all of the notation used in this paper is provided  at the end of the paper.

\section{Problem Setting}

\label{Sec:ProblemSetting}
In this section, we summarize the setting of the structured population models from \cite{düll_gwiazda_marciniak-czochra_skrzeczkowski_2021}.\\
Let $(U,d)$ be a Polish metric space, i.e.  $U$ is separable and complete with respect to the metric $d$. This space will be used as the state space of the structured population model. 

By $BL(U)=\{\psi:U\rightarrow \mathbb R\big \vert \Vert \psi\Vert_{BL}<\infty\}$ we refer to the space of bounded Lipschitz functions on $U$ with image in $\mathbb{R}$. The space is equipped with the norm 
\begin{align*}
    \Vert \psi \Vert_{BL}:=\max\left\{ \sup\limits_{x\in U} \vert  \psi (x)\vert, \sup\limits_{x,y\in U,\ x\neq y} \frac{\vert  \psi (x)- \psi (y)\vert}{d(x,y)}\right\}.
\end{align*}
Throughout this paper, this space will be used as a test function space.\\
Denote by $\mathcal{M}^+(U)$ the cone of non-negative Radon measures on $U$ with finite total mass. 
We equip $\mathcal{M}^+(U)$ with the flat norm, which is given by
\begin{align}
    \label{def:flatNorm}
    \Vert \mu\Vert_{BL^*}=\sup_{ \Vert \psi\Vert_{BL}\leq 1}\  \int \limits_U \psi(x)\mathrm{d}\mu(x)
\end{align}
and denote the corresponding flat metric by $\rho_F$. The space $(\mathcal{M}^+(U),\rho_F)$ is a metric space but lacks the linear structure of a vector space. This set is used to define the solutions of the models.\\
As $(U,d)$ is separable by assumption, there exists a countable dense subset $(z_l)_{l\in\mathbb{N}}$ of $U$. The rational linear combinations
\begin{align}
\left\{\sum_{l=1}^L q_l\delta_{z_l}\vert L\in \mathbb N, q_1,...,q_L\in\mathbb{Q}_{\geq 0}\right\}
\end{align}
form a countable dense subset of the space $(\mathcal{M}^+(U),\rho_F)$ \cite[Theorem 1.37]{düll_gwiazda_marciniak-czochra_skrzeczkowski_2021}. \\
Next, we define a suitable notion of continuity. 
A family of measures $(\mu(t))_{t\in[0,T]}\subset\mathcal M^+(U)$ is narrowly continuous if for all $t\in[0,T]$ and all $\psi\in C^0_b(U)=\big\{ \varphi \in C^0(U)\big\vert \Vert\varphi\Vert_\infty<\infty\big\}$, it holds that
\begin{align*}
    \lim_{s\rightarrow t}\int_U \psi(x)\mathrm{d}[\mu(s)](x)=\int_U \psi(x)\mathrm{d}[\mu(t)](x).
\end{align*}
Narrow continuity is metricized on $\mathcal M^+(U)$ by the flat metric  \cite[Corollary 1.47]{düll_gwiazda_marciniak-czochra_skrzeczkowski_2021}.\\
For a Borel-measurable function $f:U\rightarrow U$ and a measure $\mu\in \mathcal{M}^+(U)$, the push-forward measure is denoted by $f_\#\mu$, i.e. for a measurable set $A\subset U$,  $f_\#\mu(A):=\mu(f^{-1}(A))$.\\

Now we are in a position to consider model \eqref{Eq:MeasurePDE} on a Polish metric state space. As discussed in \cite{dull2024}, the lack of a linear structure requires a formulation without derivatives. 
We study the nonlinear structured population model with solutions that are Radon measures. The solutions are defined as measures on metric spaces and are  described by push forwards according to the theory developed in \cite[Chapter 3]{düll_gwiazda_marciniak-czochra_skrzeczkowski_2021}.
We formulate the structured population model using the following model functions:
\begin{align}
&c: [0,T]\times U\times\mathcal{M}^+(U)\rightarrow \mathbb{R}\\
&N:[0,T]\times\mathcal{M}^+(U)\rightarrow \mathcal{M}^+(U)\\
&X:[0,T]\times[0,T]\times U\times([0,T]\rightarrow\mathcal{M}^+(U))\rightarrow U\\
&\eta:[0,T]\times U\times\mathcal{M}^+(U)\rightarrow\mathcal{M}^+(U).
\end{align}
\begin{definition}
\label{structuredPopulationModel}
A narrowly continuous map $\mu(\cdot):[0,T]\rightarrow \mathcal{M}^+(U)$ is called a solution to the nonlinear structured population model with initial value $\mu(0)$ on a Polish metric space $(U,d)$  if it satisfies
\begin{align}
\label{eq:mod_pushforward}
\begin{split}
\mu(t)=
&X(t,0,\cdot,\mu(\cdot))_\#\bigg(\mu(0)(\cdot)e^{\int\limits_0^t c(s,X(s,0,\cdot,\mu(\cdot)),\mu(s))\mathrm{d}s}\ \bigg)\\
&+\int\limits_0^t X(t,\tau,\cdot,\mu(\cdot))_\#\left(\int\limits_U [\eta(\tau,x,\mu(\tau))(\cdot)]\mathrm{d}[\mu(\tau)](x)e^{\int\limits_\tau^t c(s,X(s,\tau,\cdot,\mu(\cdot)),\mu(s))\mathrm{d}s}\right)\mathrm{d}\tau\\
&+\int\limits_0^t X(t,\tau,\cdot,\mu(\cdot))_\# \left(N(\tau,\mu(\tau))(\cdot)e^{\int\limits_\tau^t c(s,X(s,\tau,\cdot,\mu(\cdot)),\mu(s))\mathrm{d}s}\right)\mathrm{d}\tau.
\end{split}
\end{align}
The solution $\mu(t)$ describes the distribution of individuals on the state space $U$ at time $t$. \\
\end{definition}

Here, $X$ denotes the transport through the state space and it generalizes the flow of a vector field, see Remark \ref{AssumptionsX}. More precisely, A mass that was at a coordinate $x_0 \in U$ at time $t_0$ is transported such that it is located at $X(t_1,t_0,x_0,\mu(\cdot))\in U$ at time $t_1$. For a more detailed motivation for the nature of the model function $X$, see \cite[Chapter 3]{düll_gwiazda_marciniak-czochra_skrzeczkowski_2021}.
The real valued model function $c$ represents growth and decay.
Further, the measure $N$ is an influx function. Lastly, the measure $\eta$ describes non-local transitions in the state space, i.e. potentially discontinuous changes of the state variable, which do not have to be mass conserving. 
One  example for $\eta$  on the state space $\mathbb{R}^+$ is  cell division: one cell of size $x$ creates 2 cells of size $\frac{x}{2}$. 
See \cite[Chapter 5]{düll_gwiazda_marciniak-czochra_skrzeczkowski_2021} for further examples how this non-local function can be used in applications.
\\

 Existence and uniqueness of solutions to model \eqref{eq:mod_pushforward} are guaranteed by the results in \cite{düll_gwiazda_marciniak-czochra_skrzeczkowski_2021} under the assumptions listed in Appendix \ref{Appendix:Assumptions}. It can even be shown that solutions of \eqref{Eq:MeasurePDE} and \eqref{eq:mod_pushforward} coincide on the Euclidean space $\mathbb{R}^d$  \cite[Section 3.4]{düll_gwiazda_marciniak-czochra_skrzeczkowski_2021}.  For the convenience of the reader, a comprehensive summary can be found in Appendix \ref{Sec:SolutionsandProperties}. 
 In order to establish well-posedness and convergence of the numerical scheme, we additionally assume several conditions on the regularity  of the model functions. We state the detailed conditions in Appendix \ref{Appendix:Assumptions} and only give an overview here:
 \begin{itemize}
     \item $N$ and $\eta$ are tight i.e. most of the mass is concentrated within a compact set,
     \item The model functions $c$, $N$ and $\eta$ are  $\alpha$-Hölder continuous in time for some $\alpha\in(0,1]$ and uniformly Lipschitz continuous in the state variable and in the measure argument,
     \item  $c$, $N$ and $\eta$ have uniformly bounded norms,
     \item $c$ is assumed to be $C^1$ regular in time. If $\eta$ and $N$ are linear, we need the following conditions
 
       \begin{align}
       &\int\limits_U \psi(y)\mathrm d [\eta(t,x)](y) \in C^1([0,T])\text{ and }\\
       &\int\limits_U \psi(y)\mathrm d [N(t)](y) \in C^1([0,T])\\
       \end{align}
    for all test functions $\psi\in BL(U)$. If the functions are nonlinear, we need to assume that they have a finite range and that the coefficients are $C^1$ regular in time, see Assumption \ref{assumptionsOnModelFunctions} \eqref{Ass:finiteRangeNonlinear}.
    \item The transport function $X$ is uniformly continuous in time and Lipschitz continuous in both the state variable and in the measure argument. Furthermore,  $X$ fulfills a semigroup property.
     
 \end{itemize}

\section{The numerical scheme}
\label{Sec:NumericalScheme}
 We approximate the solutions to \eqref{eq:mod_pushforward} by linear combinations of Dirac measures, which we can imagine as particles of various masses moving through the state space. The simulation of the evolution of the particles is based on splitting the model into two operators.  The first one accounts for changes in the state variable due to the transport function $X$ whereas the second one describes the effects of $c$, $N$ and $\eta$, that change the mass of the particles and generate new particles.\\

According to Theorem \ref{Thm:solution}, the exact solution $(\mu(t))_{t\in[0,T]}$ of \eqref{eq:mod_pushforward} is tight, i.e. for any given $r>0$, we can find a compact set $K\subset U$ such that 
\begin{align}
\label{Eq:tightnessSolution}
    [\mu(t)](U\backslash K)\leq r
\end{align}
for all $t\in[0,T]$.\\
For the remainder of this paper, we fix r and restrict our analysis to the corresponding compact set K to capture the most important aspects of the dynamics. In particular, we approximate the model functions and solutions on K or rather on an open covering of K.

\phantomsection
\label{Paragraph:Dense_subset}
Due to the compactness $K$, the set can be covered by finitely many balls of radius $\varepsilon$ with center points $\{z_l\ \vert l=1,...,L\}\subset U$. The parameter $\varepsilon$ determines the accuracy of the approximations and the total number of support points.\\
 We will approximate the model functions $\eta$ and $N$ as well as the initial measure $\mu(0)$ on the set $\mathcal{K}=\bigcup_{l=1,...,L} B_{\varepsilon}(z_l)$. We use a partition of unity to decompose the model functions without losing the continuity in the state variable $x$. For our construction, we require the partition of unity to sum up to 1 on $K$ while simultaneously being Lipschitz continuous on the entire space $U$. This is why the partition of unity has support in $\mathcal K$ and not only in $K$.
\begin{theorem}
[Discrete approximation of $\eta$ and $N$]\ 

   \label{Theorem:finiteRangeApproximation}
   Assume that $N$ and $\eta$ are linear, i.e. do not depend on $\mu$.
   Denote by $(B_\varepsilon(z_l))_{l=1,...,L}$ an open covering of $K$ and let $(\varphi_l)_{l=1,...,L}$ be the corresponding partition of unity consisting of Lipschitz continuous functions, which exists due to Lemma \ref{Lemma:PartiotionOfUnity}. Let us define the approximations
   \begin{align}
       &\hat\eta(t,x)(\cdot)=\sum_{l=1}^L 
      \left(\int\limits_U \varphi_l(y)\mathrm{d}[\eta(t,x)](y)\right) \delta_{z_l}(\cdot) \text{ and }\\
       &\hat N(t)(\cdot)=\sum_{l=1}^L\left( \int\limits_U \varphi_l(y)\mathrm{d}[N(t)](y)\right)  \delta_{z_l}(\cdot).
   \end{align}
   The functions $\hat N$ and $\hat\eta$ fulfill Assumptions \ref{assumptionsOnModelFunctions} \eqref{assumptionsSupremum}and \eqref{assumptionsOnModelFunctionsTime}. Further, the approximation errors are at most
   \begin{align}
   \label{Eq:Approximationerror}
   &\rho_F\big(\eta(t,x),\hat\eta(t,x)\big)\leq (\varepsilon+r)  \Vert \eta(t,x)\Vert_{BL^*}\text{ and }\\
       &\rho_F\big(N(t),\hat N(t)\big)\leq (\varepsilon+r) \Vert N(t)\Vert_{BL^*},
   \end{align} 
   where $r$ is the cut-off error from \eqref{Eq:tightnessSolution}.
\end{theorem}

We postpone the proof to Section \ref{Sec:approximation}.

\begin{remark}
Due to Theorem \ref{Theorem:finiteRangeApproximation} and Theorem \ref{Thm:solution}, there exists a solution to \eqref{eq:mod_pushforward} with model functions $\hat \eta$ and $\hat N $. We denote this solution by $\hat \mu$.
\end{remark}
We are now in the position to propose a numerical scheme to approximate $\hat \mu$.\\

\noindent \textbf{Approximation of the initial measure}\hfill\\
\label{ApproximationInitial}First, we approximate the initial value $\mu(0)\in\mathcal M^+(U)$ by a linear combination of Dirac measures. The approximation only takes into account the mass in $K\subset U$ due to the tightness of the solutions as explained in the beginning of this section.\\

The initial measure $\mu(0)$ is approximated by 
\begin{align*}
\mu_0(\cdot)=\sum_{l=1}^L{\mu}(U_{\varepsilon,l})\delta_{z_l}(\cdot),
\end{align*}
where  $U_{\varepsilon,l}$ refers to the set $\left(B_\varepsilon(z_l)\backslash\left(\bigcup\limits_{i<l}B_\varepsilon(z_i)\right)\right)\cap K$. As the sets are a disjoint cover of $K$, we conserve the total mass of $\mu(0)$ within $K$.\\

\noindent\textbf{Approximation of growth and transport}\\
\label{grothAndTransport}\noindent For the algorithm, we fix a step size $\Delta t$.
Assume that at time $k\Delta t$ the measure is given as a linear combination of $J_k$ Dirac measures with mass $m^j_k$ located at supporting points $x^j_k$, i.e.
\begin{align*}
     \mu_{k}=\sum_{j=1}^{J_k} m^j_k\delta_{x^j_k}.
\end{align*}
We explain how to execute one time step from $k\Delta t$ to $(k+1)\Delta t$ which consists of two stages. At first, we only account for the transport which means that only the locations of the Dirac masses are changed and the corresponding masses remain constant. In the second part, we consider the changes due to $c$, $\hat N$ and $\hat \eta$. Here, the weights of the Dirac measures are adjusted and new support points are created if necessary.

\noindent\textbf{Transport step}\\
\noindent The locations at time $(k+1)\Delta t$ are set to
\begin{align}
    x^{j}_{k+1}=X((k+1)\Delta t,k\Delta t,x^j_k,\mu_k).
    \label{trasnportNumericalScheme]}
\end{align}
Note that $X$ may also be the identity function, i.e. it is possible that $x^j_{k+1}=x^j_k$.

\begin{remark}
    For simplicity, we assume here that the transport function $X$ is known explicitly. In the case that $U\subset \mathbb{R}^d$, the transport function $X$ might be given as the solution to an ordinary differential equation, see Remark \ref{Appendix:Assumptions}, which would need to be solved in this step.
\end{remark}

\textbf{Growth and non-local interaction}\\
Let $\hat \eta$ and $\hat N$ be of the following form:
\begin{align}
    &\hat\eta(t,x,\mu)(\cdot)=\sum\limits_{l=1}^L \beta_l(t,x,\mu)\delta_{z_l}(\cdot)
    \text { and}\\
    & \hat N(t,\mu)(\cdot)= \sum\limits_{l=1}^L n_l(t,\mu) \delta_{z_l}(\cdot),
\end{align}
which means that new masses are only created in $\{z_l\ \vert l=1,...,L\}$. This is important for the algorithm as it means that in each time step only a finite number of new support points is created. Note that according to Theorem \ref{Theorem:finiteRangeApproximation} any linear $\eta$  and $N$ can be approximated in this discrete form. For nonlinear $\eta$ and $N$, such an approximation is not available yet without further assumptions on the regularity and compromises  to the order of convergence, see Remark \ref{Remark:issuesNonLinear}, so that in the nonlinear case we assume that $\eta$ and $N$ are already of the above form. For ease of notation, we will still write $\hat\eta$ and $\hat N$ in this case.

Due to the form of the model functions $\hat\eta$ and $\hat N$,
the set of points with possibly non-zero mass at time $(k+1)\Delta t$ is given by
\begin{align}
\label{stesGrowth}
    \{x_{{k+1}}^j\vert j=1,...,J_k\}
    \cup
    \{z_l\ \vert {l=1,...,L}\},
\end{align}
as the only points with non-zero mass at time $(k+1)\Delta t $ fall into one of the following categories
\begin{enumerate}
    \item points with mass at time  $k\Delta t$ with locations modified by the transport function and weights modified by the model function $c$,
    \item points $\{z_l\ \vert l=1,...,L\}$ receiving mass from the influx function $\hat N$,
    \item points $\{z_l\ \vert l=1,...,L\}$ receiving mass via the model function $\hat\eta$ from a point which had mass at time $k\Delta t$.   
\end{enumerate}
We renumber the points to obtain
\begin{align}
   \{x_{{k+1}}^j\vert j=1,...,J_k\}
    \cup
   \{z_l\ \vert l=1,...,L\}=\{\xi^1_{k+1},...,\xi^{J_{k+1}}_{k+1}\},
\end{align}
where $J_{k+1}\leq J_k+L$.
For each point 
$\xi_{k+1}^j$,
we solve an explicit Euler equation for the evolution of the mass $m_\xi$:
\begin{align}
\label{MassEvolution}
    \frac{m^j_{k+1}-m^j_k}{\Delta t}
    =&
    c(k\Delta t,x^{\bar i}_{k+1},{\mu}_k)m^{\bar i}_k \mathbbm{1}_{\xi_{k+1}^j=x^{\bar i}_{k+1}}\\
    &+\bigg(\bigg(\sum_{i=1}^{J_k}\beta_l(k\Delta t,x^i_{k+1},\mu_k)m^i_{k}\bigg)
    + n_l(k\Delta t,\mu_k)\bigg)\mathbbm{1}_{\xi_{k+1}^j=z_l},
\end{align}
where we set 
\begin{align}
m_j^k=
    \begin{cases}
    0 &\quad \text{ if } \xi_{k+1}^j\notin \{x_{{k+1}}^j\vert j=1,...,J_k\},\\
    m^{\bar i}_k & \quad \text{ if }\xi_{k+1}^j=x^{\bar i}_{k+1}  \text{ for some } \bar i\in \{1,...,J_k\}.
    \end{cases}
\end{align}
Overall, we obtain the approximation
\begin{align*}
    \mu_{k+1}=\sum_{j=1}^{J_{k+1}} m^j_{k+1} \delta_{\xi^j_{k+1}}. 
\end{align*}

The above splitting approach leads to the following convergence result.

\begin{theorem}[Convergence of the numerical scheme] \ 
\label{Thm:ConvergenceNumerical}

    The overall error of the numerical scheme is
    \begin{align}
        \rho_F\big(\mu_k,\mu(k\Delta t)\big)
        \leq C\big( \varepsilon+r+\rho_F\big(\mu_0,\mu(0)\big)
        + [\Delta t+(\Delta t)^\alpha]\big),
    \end{align}
where $\Delta t$ is the size of the time steps and $r$ is a cutoff parameter of the initial values and model functions, see \eqref{Eq:tightnessSolution} and \eqref{Eq:Approximationerror}. The parameter $\varepsilon$ represents the  density of the grid points (see  Section \ref{Paragraph:Dense_subset}) and $\alpha$ refers to the assumed continuity of the model functions w.r.t. time (see Assumptions \ref{assumptionsOnModelFunctions}\eqref{assumptionsOnModelFunctionsTime}). 

\end{theorem}
\begin{remark}
    Throughout this paper, all of the constants that we estimate may depend continuously  on $T$, $\mu_0$ and on the norms of the model functions $c$, $N$, $\eta$ and $X$ and their approximations $\hat N$ and $\hat\eta$. 
\end{remark}

The error in Theorem \ref{Thm:ConvergenceNumerical} can be controlled by decreasing the parameters $r$ , $\Delta t$ and $\varepsilon$, which improves the accuracy of the algorithm.
 The distance between the support points of the approximation is determined by $\varepsilon$. A smaller $\varepsilon$ requires more support points, also increasing the  computational cost. On the other hand, it also decreases the  initial error given by $\rho_F\big(\mu(0),\mu_0\big)$ (see Lemma \ref{IntialError}) and the error in the approximation of $N$ and $\eta$ (see Theorem \ref{Theorem:finiteRangeApproximation}).\\
 Similarly, decreasing $r$ implies that the set on which we approximate the measures grows (see \eqref{Eq:tightnessSolution}) and more points are needed to cover the set. Lastly, decreasing $\Delta t$ leads to higher computational costs as the number of steps needed to reach a set end time increases.\\
 The error of order $\Delta t$ comes from the splitting of the process into the transport and the growth step. The error of order $(\Delta t)^\alpha$ arises from the freezing of the model functions at the beginning of each time step.

\begin{remark}
 The number of Dirac measures in the approximation grows at most linearly in the number of time steps. In each step, at most $L$ new masses are created such that the number of total support points is bounded from above by $\frac{T}{\Delta t}L $ plus the number of points in the initial approximation. Hence, the number of equations that are solved in each time step also grows linearly with the number of time steps. \\

\end{remark}

\section{Bayesian Inverse Method}
\label{Bayes}
After establishing a numerical scheme for solving the structured population model on an arbitrary Polish metric space, we turn to inverse problems and in particular to the Bayesian inverse method. \\
For the inverse problem, we want to determine the parameters governing the process that gives rise to a certain solution. The potential difficulties are that several parameters may produce the same solution or that the dependence of the parameters on the solution may not be continuous \cite{Stuart_2010}.\\
The Bayesian inverse method addresses these issues by establishing a posterior distribution on the parameter space, which quantifies how likely it is that the model was parameterized by a certain parameter $\theta$ given observed data $Y\!$.
This is done by linking the posterior distribution $\pi(\theta\vert Y)$ to a prior distribution $\pi(\theta)$ via a likelihood function $\ell(Y\vert\theta)$. The prior quantifies knowledge about the parameter space before any observations were made. The likelihood describes the possibility of observing data $Y$ if the model is parameterized by $\theta$ \cite{Stuart_2010}.\\
In this section, we consider the case where we assume that the model \eqref{eq:mod_pushforward} can be parameterized by finitely many real-valued parameters which are located in a compact set. \\
We show that the numerical scheme proposed in Section \ref{Sec:NumericalScheme} can be used to calculate a posterior density that is close to the exact posterior density in the flat metric.  
To this end, we proceed as in \cite{szymanska_bayesian_2021}.\\

\subsection{The data model and the posterior}\hfill\\
Let 
\begin{align}
\label{Eq:solutionOperator}
    \mathcal{G}:\mathcal{M}^+(U)\times \Theta \rightarrow C^0([0,T],\mathcal{M}^+(U))
\end{align} 
be the solution operator to the model given by \eqref{eq:mod_pushforward} with model functions parameterized by parameters $\theta$ in a compact set $ \Theta\subset\mathbb{R}^d$ such that the Assumptions \ref{assumptionsOnModelFunctions} are satisfied. 
The operator $\mathcal{G}$ maps an initial value $\mu(0)$ and a vector of parameters $\theta\in\Theta$ to the corresponding solution $t\mapsto\mu(t)$ of \eqref{eq:mod_pushforward}.\\

Denote by $Y$ the observable data. We assume that it is measured up to a random measurement error $Z$, which is distributed according to a known probability density $f$ and that the measured data is given at a certain number of deterministic time points $t_m$ for $m=1,...,M$. Let the observation $y_{i,m}$ at time $t_m$ be of the form   
\begin{align}
    y_{i,m}=\int_U g_i(x) \mathrm{d}[\mathcal{G}(\theta)(t_m)](x)+Z_{i,m},
\end{align} 
where  $g_i\in BL(U)$ for $i=1,...,I$ are known functions. This means that $y_{i,m}$ is the integral of $\psi_i$ with respect to the measure $\mathcal G (\theta)(t_m)$ up to the error $Z_{i,m}$. The assumed form of the data is chosen such that we can link the data and the model continuously.\\
We assume that the density $f$ of $Z$ is $\gamma$-Hölder continuous for some $\gamma\in(0,1]$, i.e. we assume that $f\in C^{0,\gamma}(\mathbb{R}^{I\cdot M},\mathbb{R})$
\begin{align}
\label{rhoHoelder}
    \vert f(X_1)-f(X_2)\vert\leq \Vert f\Vert_{C^{0,\gamma}} \Vert X_1-X_2\Vert_{\mathbb{R}^{I\cdot M}}^\gamma.
\end{align}
The commonly used multivariate normal distribution with symmetric positive definite covariance matrix is smooth with bounded derivatives and hence 1-Hölder continuous.\\ 
In order to link the model and the data, let 
\begin{align}
    \mathcal{F}:C^0([0,T],\mathcal{M}^+(U))\rightarrow \mathbb{R}^{I\cdot M}
\end{align}

 be an observation operator which connects the solution $\mathcal{G}(\theta)$ of \eqref{eq:mod_pushforward} to the observations: 
\begin{equation}
\label{eq:Data}
    Y=\mathcal{F}(\mathcal{G}(\theta))+Z
    .
\end{equation}
More precisely, for $\mu(\cdot)\in C^0([0,T],\mathcal{M}^+(U))$, the operator $\mathcal{F}$ is given by
\begin{align}
\label{Eq:defObservationOperator}
    \mathcal{F}(\mu(\cdot))_{i,m}=\int\limits_U g_i(x)\mathrm{d}[\mu({t_m})](x).
\end{align}
This observation operator is used to define our prior on a subset of  $\mathbb{R}^{I\cdot M}$. This avoids the technical difficulties of a prior measure on a space of functions with values in $\mathcal{M}^+(U)$.

\begin{remark}
The exact form of $\mathcal{F}$ is motivated by the fact that for $\psi\in BL(U)$ evaluations of the form 
\begin{align}
\label{Eq:TypeOfData}
    \int\limits_U \psi(x)\mathrm{d} [\mathcal{G}(\theta)(t)](x)
\end{align}
 are what we can explicitly calculate for solutions $\mathcal G(\theta)$ to \eqref{eq:mod_pushforward}, see \eqref{eq:Weakform}.\\
More intuitively, one would maybe consider data that is of the form
\begin{align}
\label{eq:dataSet}
    \tilde y_{i,m} =[\mathcal{G}(\theta)({t_m})](A_i)+Z_{i,m}
\end{align}
for some measurable set $A_i\subset U$. However, as the indicator functions of sets are not Lipschitz continuous, such an evaluation operator could not be Lipschitz continuous w.r.t. the flat metric.  Nevertheless, we can approximate the indicator function of an open set by a Lipschitz function up to a pre-determined accuracy, such that we can consider data of the type \eqref{eq:dataSet} at least approximately.\\
Further, the observation operator defined by \eqref{Eq:defObservationOperator} can be used to describe the first moment of the measure $\mathcal G(\theta)(t_m)$ as the function
\begin{align}
    x\mapsto \mathrm{d}(x_0,x)
\end{align}
is Lipschitz continuous for a fixed value $x_0$ \cite[Appendix C]{düll_gwiazda_marciniak-czochra_skrzeczkowski_2021}.\\

\end{remark}

\begin{lemma}
\label{ObsContinuous}
     The observation operator $\mathcal{F}$ is Lipschitz continuous.
\end{lemma}
\begin{proof}
   This follows by a direct calculation using the scaling properties of the flat metric (Lemma \ref{scalingFlatNorm}).
\end{proof}
We introduce the likelihood function $\ell$ by rearranging \eqref{Eq:likelyhood} and using that the measurement error is distributed according to the density $f$ such that
\begin{align}
\label{Eq:likelyhood}
\ell(Y\vert \theta):=f\big(Y-\mathcal{F}\circ\mathcal{G}(\theta)\big).
\end{align}
 The likelihood describes the possibility of observing data $Y$ if the model is parameterized by $\theta$.\\
\noindent
According to Bayes' theorem, the probability density of the posterior is given by 
\begin{equation}
    \pi (\theta\vert Y)=\frac{\ell (Y\vert\theta)\pi (\theta) }{\int\limits_\Theta\ell(Y\vert\theta)\pi(\theta)\mathrm{d}\theta},
\end{equation}
where $\pi$ is the prior density on the parameter space $\Theta$. Choosing which probability density to use for the prior is part of the modeling. 
As $\Theta$ is compact, we can for example choose the density of the uniform distribution as the prior.
\begin{lemma}
The posterior density depends continuously on the data:
Given two observations $Y_1,Y_2\in \mathbb{R}^{I\cdot M}$, we can estimate
\begin{align*}
    \rho_F\big(\pi(\theta\vert Y_1),\pi(\theta\vert Y_2)\big)
    \leq C\cdot \Vert f\Vert_{C^{0,\gamma}} \Vert Y_1-Y_2\Vert_{\mathbb{R}^{I\cdot M}}^\gamma \frac{1}{\int\limits_{\Theta} \ell(Y_1\vert\theta)\pi(\theta)\mathrm{d}\theta}.
\end{align*}
\end{lemma}
\begin{proof}
This follows from a direct calculation using the regularity of $f$.
\end{proof}

\subsection{Numerical approximation of the posterior}\hfill \\
In practice, the exact solution operator $\mathcal G$ is unknown, so that we are restricted to work with its approximation $\mathcal{G}_{\Delta t}$ as obtained by the algorithm proposed  in Section \ref{Sec:NumericalScheme}. This numerical approximation induces a bias in the posterior distribution, which we estimate in this subsection.\\
As established in Theorem \ref{Thm:ConvergenceNumerical}, the numerical approximation $\mathcal{G}_{\Delta t}$ of $\mathcal{G}$ fulfills

 \begin{align}
 \label{consistencyNumerical}
        &\rho_F\big(\mathcal{G}_{\Delta t}(\theta)(t),\mathcal{G}(\theta)(t)\big)
    \leq  C\big( \varepsilon+r+\rho_F\big(\mu_0,\mu(0)\big) 
   + \left[ (\Delta t)+(\Delta t)^{\alpha}\right]\big)
    \end{align}
  
    for all $t\in[0,T]$. Here, $\mu(0)$ refers to the exact initial value and $\mu_0$ refers to its approximation by a linear combination of Dirac measures.\\
    The data $Y$ is still determined by the exact model. For a parameter $\theta\in\Theta$ we aim to find the probability that
    \begin{align}
        Y=\mathcal{F}(\mathcal{G}_{\Delta t}(\theta))+Z
    \end{align}
    for predetermined accuracy parameters $\Delta t, \varepsilon, r$ of the numerical scheme.\\
    
We define the functions $\pi_{\Delta t}(\theta\vert Y)$ and $\ell_{\Delta t}(Y\vert \theta)$ by the same method as $\pi(\theta\vert Y)$ and $\ell(Y\vert \theta)$ using $\mathcal{G}_{\Delta t}$ instead of $\mathcal{G}$ and aim to compare the two posterior measures with densities $\pi (\theta\vert Y)$ and $\pi_{\Delta t} (\theta\vert Y)$.
\begin{theorem}[Stability of the posterior measure w.r.t. numerical approximation]\ \\
The difference between the exact posterior and the posterior obtained using the numerical approximation is of the same order of accuracy as the numerical approximation:

    \begin{align}
      \rho_F\big(\pi(\theta\vert Y),\pi_{\Delta t}(\theta\vert Y)\big) 
      \leq C\big(\varepsilon+r+\rho_F\big(\mu_0,\mu(0)\big) 
   +\left[ (\Delta t)+(\Delta t)^{\alpha}\right]) \big)^\gamma \frac{1}{\int\limits_{\Theta} \ell(Y_1\vert\theta)\pi(\theta)\mathrm{d}\theta}.
    \end{align}

\end{theorem}
\begin{proof}
First, we estimate the difference between the exact likelihood $\ell$ and the approximate likelihood $\ell_{\Delta t}$ for a fixed $\theta\in \Theta$:
\begin{align*}
    &\vert \ell(Y\vert \theta)-\ell_{\Delta t}(Y\vert\theta)\vert
    =
   \big \vert f\big(Y-\mathcal{F}(\mathcal{G}(\theta)\big)- f\big(Y-\mathcal{F}(\mathcal{G}_{\Delta t}(\theta)\big)\big\vert
    \\
    &\leq
    \Vert f\Vert_{C^{0,\gamma}} \Vert Y-\mathcal{F}(\mathcal{G}(\theta))-Y-\mathcal{F}(\mathcal{G}_{\Delta t}(\theta))\Vert_{\mathbb{R}^{M\cdot I}}^\gamma
    \\
    &=
    \Vert f\Vert_{C^{0,\gamma}} \Vert \mathcal{F}(\mathcal{G}(\theta))-\mathcal{F}(\mathcal{G}_{\Delta t}(\theta)) \Vert_{\mathbb{R}^{M\cdot I}}^\gamma
    \\
    &\leq
    C \sup_{t\in[0,T]} \Vert \mathcal{G}(\theta)(t)-\mathcal{G}_{\Delta t}(\theta)(t) \Vert_{BL^*}^\gamma
    \\
    &\leq C\big(\epsilon+r+\rho_F\big(\mu_0,\mu(0)\big) 
   +\left[ (\Delta t)+(\Delta t)^{\alpha}\right]\big)^\gamma,
\end{align*}
where we used the continuity of the noise density $f$ \eqref{rhoHoelder} and of the observation operator $\mathcal{F}$ due to Lemma \ref{ObsContinuous} as well as the consistency of the numerical scheme \eqref{consistencyNumerical}.\\
As $\Theta$ is compact, the constant $C$ can be chosen uniformly for all $\theta$.\\

For any test function $\psi\in BL(\Theta)$ with $\Vert\psi\Vert_{BL(\Theta)}\leq 1$, we can estimate the difference between the exact and the numerical posterior by
\begin{align}
   &\int\limits_\Theta \psi(\theta)\big(\pi(\theta \vert Y)-\pi_{\Delta t}(\theta\vert Y)\big)\mathrm{d}\theta\\
   &=\frac{1}{\int\limits_\Theta \ell(Y\vert\theta)\pi(\theta)\mathrm{d}\theta} \int\limits_\Theta \psi(\theta)\big(\ell(Y\vert\theta)-\ell_{\Delta t}(Y\vert\theta)\big) \pi(\theta)\mathrm{d}\theta\\
   &\quad+ \frac{\int\limits_\Theta \big(\ell_{\Delta t}(Y\vert\theta)-\ell(Y\vert\theta)\big)\pi(\theta)\mathrm{d}\theta}{\int\limits_\Theta \ell(Y\vert\theta)\pi(\theta)\mathrm{d}\theta\cdot\int\limits_\Theta \ell_{\Delta t}(Y\vert\theta)\pi(\theta)\mathrm{d}\theta}\int\limits_\Theta \psi(\theta)\ell_{\Delta t}(Y\vert\theta)\pi(\theta)\mathrm{d}\theta \\
   &\leq \frac{1}{\int\limits_\Theta \ell(Y\vert\theta)\pi(\theta)\mathrm{d}\theta} \int\limits_\Theta C\big(\varepsilon+r+\rho_F\big(\mu_0,\mu(0)\big) 
   +\left[ (\Delta t)+(\Delta t)^{\alpha}\right]\big)^\gamma\pi(\theta)\mathrm{d}\theta\\
   &\quad+ \frac{\int\limits_\Theta C\big(\varepsilon+r+\rho_F\big(\mu_0,\mu(0)\big) 
   +\left[ (\Delta t)+(\Delta t)^{\alpha}\right]\big)^\gamma \pi(\theta)\mathrm{d}\theta}{\int\limits_\Theta \ell(Y\vert\theta)\pi(\theta)\mathrm{d}\theta\cdot\int\limits_\Theta \ell_{\Delta t}(Y\vert\theta)\pi(\theta)\mathrm{d}\theta}\int\limits_\Theta \ell_{\Delta t}(Y\vert\theta)\pi(\theta)\mathrm{d}\theta\\
   &\leq C\big(\epsilon+r+\rho_F\big(\mu_0,\mu(0)\big) 
   +\left[ (\Delta t)+(\Delta t)^{\alpha}\right]) \big)^\gamma \frac{1}{\int\limits_{\Theta} \ell(Y\vert\theta)\pi(\theta)\mathrm{d}\theta}.
\end{align}

By taking the supremum over all test functions, we obtain the desired result.
\end{proof}

Hence, the numerical scheme proposed in this paper can be used to approximate the posterior measure of a Bayesian inverse problem for sufficiently regular measurement data as defined in \eqref{Eq:TypeOfData}.

\section{Approximation of model functions and of the initial measure}
\label{Sec:approximation}
We need to approximate the initial measure $\mu(0)$ and the measure valued model functions $\eta$ and $N$ by finite linear combinations of Dirac measures. To this end, we need to fix a grid (an $\varepsilon$-dense subset of points) on the space $(U,d)$. While we assume that the underlying metric space is separable and hence a finite $\varepsilon$-dense subset $\{z_l\ \vert l=1,...,L\}$ of the compact subset $K\subset U$ as fixed in Section \ref{Sec:NumericalScheme} exists, it is not always clear how to construct this set.

    The number of support points needed to approximate the initial measure highly depends on the exact nature of the set $K$. The covering number $N(\epsilon,U)$ is the minimal number of balls of a certain radius $\varepsilon$ needed to cover a metric space $U$. The scaling of this covering number (or the closely related packing number and metric entropy) depends on the metric space.  If we consider for example the unit ball $B_1\subset \mathbb{R}^d$ with respect to the Euclidean distance, the covering number grows exponentially in the dimension $d$ \cite{Wainwright}.

In many spaces that are interesting for applications, we can explicitly construct a finite dense subset:
\begin{enumerate}
    \item For any bounded subset $\Omega\subset [-R,R]^d$, we can discretize the hypercube $[-R,R]^d$ by the d-fold product of $\{-R,-R+\frac{ \varepsilon}{ \sqrt{d}} ,..., r-\frac {\varepsilon}{\sqrt{d}},R\}$ and take the points that lie within an $\varepsilon$-neighborhood of $\Omega$ as the discretization.
    \item For a graph, we can discretize each edge by discretizing the interval $[0,\ell]$ where $\ell$ is the length of the edge.
    \item It is also possible to construct a grid for a geodesically complete $C^1$ manifold. A geodesically complete manifold is a complete metric space by the Hopf-Rinow Theorem  \cite{Hopf1931}. $C^1$-manifolds admit triangulations which means that they are homeomorphic to a simplicial complex \cite{Whitehead}. Using barycentric subdivision, the faces of the simplicial complex can be refined \cite{MunkresAlgebraic} such that the corners of the simplicial complex cover the manifold up to any prescribed density.
\end{enumerate}

\subsection{Error due to approximation of the measure-valued model functions}\ \\
\noindent We explain here in more detail the approximation of the measure valued model functions $N$ and $\eta$ by a finite sum of weighted Dirac measures if they are linear.

\begin{lemma}[Partition of unity in metric spaces]\ 

\label{Lemma:PartiotionOfUnity}
\noindent 
Let $K$ be a compact subset of $(U,d)$ and fix $\varepsilon>0$. 
There exists a finite index set $\{1,...,L\}$, a Lipschitz continuous partition of unity $(\varphi_l)_{l=1,...,L}$ and a set of points $\{z_l\ \vert l=1,...,L\}$ such that
\begin{enumerate}
\item $\varphi_l:U\rightarrow [0,1]$,
\label{property1}
    \item $\sum_{l=1}^L \varphi_l(x)=1$ for all $x\in K$, 
    \label{property2}
    \item for $l=1,...,L$ $\varphi_l\in BL(U)$, $supp(\varphi_l)\subset B_\varepsilon(z_l)$ and  $d(z_l,K)\leq \varepsilon$,
    \label{property3}
    \item for all $x\in U$: $0\leq\sum_{l=1}^L\varphi_l(x)\leq 1$. \label{sumBounded}
\end{enumerate}

\end{lemma}
\begin{proof}

Due to the separability of $U$, for $\varepsilon>0$ there exist $\{z_l\ \vert l\in\mathbb N\}$ that fulfills 
\begin{align}
    U=\bigcup_{l\in\mathbb N} B_\varepsilon(z_l).
\end{align}
As $K$ is compact, we can extract a finite subcover of $K$.\\
Denote by $\{1,...,L\}$ be the finite index set such that $B_\varepsilon(z_l)\cap K\neq\emptyset$.\\
    According to \cite[Theorem 2.6.5]{Cobza2019} there exists a Lipschitz continuous partition of unity for the entire space $U$ subordinate to the cover $(B_\epsilon(z_l))_{l\in\mathbb N}$, denoted by $(\varphi_l)_{l\in\mathbb N}$.
    All of these functions map from $U$ to $[0,1]$.\\
     By taking only the functions $(\varphi_l)_{l=1,...,L}$, whose support intersects with $K$, we obtain a  partition of unity of $K$ with functions that are Lipschitz continuous on all of $U$. This construction implies properties \eqref{property2} and \eqref{property3}. 
    Property \eqref{sumBounded} holds immediately, as we only sum over a subset of the partition of unity for the whole space.
\end{proof}

We use the partition of unity to construct approximations of the model functions $\eta$ and $N$.\\
Recall that by Assumption \ref{assumptionsOnModelFunctions}\ref{Ass:ModelFunctionsTight} $\eta$ and $N$ are tight, i.e. the majority of the mass lies within the compact set $K$.\\
Now, we are ready to proof Theorem \ref{Theorem:finiteRangeApproximation}.

\begin{proof}[Proof of Theorem~\ref{Theorem:finiteRangeApproximation}]
Recall that    
\begin{align}
       &\hat\eta(t,x)(\cdot)=\sum_{l=1}^L 
      \left(\int\limits_U \varphi_l(y)\mathrm{d}[\eta(t,x)](y)\right) \delta_{z_l}(\cdot) \text{ and }\\
       &\hat N(t)(\cdot)=\sum_{l=1}^L\left( \int\limits_U \varphi_l(y)\mathrm{d}[N(t)](y)\right)  \delta_{z_l}(\cdot).
   \end{align}

    The approximation error can be calculated as follows. Let $\psi\in BL(U)$ with $\Vert \psi \Vert_{BL}\leq 1$. Using that $\mathrm{supp}(\varphi_l)\subset B_\varepsilon(z_l)$, we obtain
    \begin{align}
        &\int\limits_U \psi(y)\mathrm{d}\big[\hat\eta(t,x)-\eta(t,x)\big](y)\\
        &=  \int\limits_U \psi(y)\mathrm{d}[\hat\eta(t,x)](y)
        -  \int\limits_U \sum_{l=1}^L \varphi_l(y)\psi(y)\mathrm{d}[\eta(t,x)](y)
        \\
        & \quad -\int\limits_{U} \big(1-\sum_{l=1}^L \varphi_l(y)\big)\psi(y)\mathrm{d}[\eta(t,x)](y)\\
        &= \sum_{l=1}^L\int\limits_U \varphi_l(y) \big(\psi(z_l)-\psi(y)\big)\mathrm{d}[\eta(t,x)](y)
        +\int\limits_U\big( 1-\sum_{l=1}^L\varphi_l(y)\big)\psi(y)\mathrm{d}[\eta(t,x)](y)\\
        &\leq \sum_{l=1}^L \varepsilon \int\limits_U \varphi_l(y)\mathrm{d}[\eta(t,x)](y) +[\eta(t,x)](U\backslash K)\\
       & \leq \varepsilon\Vert \eta(t,x)\Vert_{BL^*}+r\Vert \eta(t,x)\Vert_{TV}.
    \end{align}
    Here, we used that within the support $B_\epsilon(z_l)$ of $\varphi_l$, the distance of any two points is at most $\varepsilon$. Due to Lemma \ref{Lemma:PartiotionOfUnity} \eqref{property3} and \eqref{sumBounded}, $\big( 1-\sum_{l=1}^L\varphi_l(y)\big)\psi(y)$ vanishes within $K$ and is at most 1 outside of $K$.
    By taking the supremum over all test functions, we obtain that
    \begin{align}
        \rho_F\big(\hat\eta(t,x),\eta(t,x)\big)\leq (\varepsilon+r) \Vert \eta(t,x)\Vert_{BL^*}.
    \end{align}
    Note that for non-negative measures, the norms $\Vert\cdot\Vert_{BL^*}$ and $\Vert\cdot\Vert_{TV}$ coincide, see Theorem \ref{propertiesM+}.
With a similar reasoning, we can establish a bound on the approximation error for $\hat N$.\\

    For $t\in[0,T]$ and $x\in U$, we bound the flat norm of $\hat\eta(t,x)$. 
    By the triangle inequality, we obtain that
    \begin{align}
        \Vert \hat\eta(t,x)\Vert_{BL^*}
        \leq \Vert \hat\eta(t,x)-\eta(t,x)\Vert_{BL^*}+ \Vert \eta(t,x)\Vert_{BL^*}
        \leq (\varepsilon+r+1)\Vert\eta(t,x)\Vert_{BL^*}.
    \end{align}
       Consequently, the assumptions in \ref{assumptionsOnModelFunctions} involving estimates or integrals of this norm follow immediately from the assumptions on $\eta$.
       \\
    Next, we need to verify the narrow continuity of $\hat \eta$ in $x$. 
    For a test function $\psi\in C^0_b(U)$, we obtain for $d(x_1,x_2)\rightarrow 0 $:

    \begin{align}
        \int\limits_U \psi(y)\mathrm{d}[\hat\eta(t,x_1)-\hat\eta(t,x_2)](y)
        &= \sum_{l=1}^L \psi(z_l)\int\limits_U \varphi_b(z)\mathrm{d}(\eta(t,x_1)-\eta(t,x_2))(z)\rightarrow 0
    \end{align}
    where we used the narrow continuity of $\eta$ in $x$, see \ref{assumptionsOnModelFunctions} \eqref{assumptionsSupremum}. 

\end{proof}
\begin{remark}
\label{Remark:issuesNonLinear}
    Note that the Lipschitz constants of the partition of unity depend on the radius $\varepsilon$ and the geometry of the covering of $K$. For any $x\in K$, it holds that
    \begin{align}
        1
        =\sum_{l=1}^L \varphi_l(x)
        \leq \sum_{l=1}^L \Vert \varphi_l\Vert_\infty
        \leq \varepsilon \sum_{l=1}^L \vert \varphi_l\vert_{Lip}
    \end{align}
    as $\varphi_l$ vanishes on the boundary of $B_\varepsilon(z_l)$.
    Hence, the sum of the Lipschitz constants grows as follows
    \begin{align}
        \frac 1 \varepsilon\leq \sum_{l=1}^L \vert \varphi_l\vert_{Lip}.
    \end{align}
    This prevents us from approximating a nonlinear $N$ or $\eta$ with the approach suggested above, as the approximations would be Lipschitz continuous w.r.t. a fixed covering with balls of radius $\varepsilon$ but the Lipschitz constant would explode as $\varepsilon\rightarrow 0$.
    While for any fixed value of $\varepsilon$,  the conditions for the well posedness of \eqref{eq:mod_pushforward} are fulfilled (see \ref{assumptionsOnModelFunctions}), all of the constants in the estimates may depend on the Lipschitz constants of $\hat N$ and $\hat \eta$ and we would hence not be able to control the error.\\
    When trying to work without the Lipschitz continuity of $\hat\eta$ and $\hat N$ w.r.t. the measure argument, we cannot use the arguments from \cite[Chapter 3]{düll_gwiazda_marciniak-czochra_skrzeczkowski_2021} for the existence of a solution $\hat\mu$ to \eqref{eq:mod_pushforward} with model functions $\hat\eta$ and $\hat N$ which an essential part of the proofs in Section \ref{Sec:convergence}.
\end{remark}

\subsection{Error due to approximation of the initial measure} \hfill\

\begin{lemma}[Error from the initial approximation]\ 

\label{IntialError}
Denote by 
\begin{align}
    \mu_0
    =\sum_{l=1}^L [\mu(0)](U_{\varepsilon,l})\delta_{z_l}
\end{align}
the approximation of the initial measure $\mu(0)$ (see Section \ref{ApproximationInitial}).
    The error from this approximation in the flat metric is
    \begin{align*}
        \rho_F\big(\mu(0),\mu_0\big)\leq (\varepsilon+r)\Vert \mu(0)\Vert_{TV}.
    \end{align*}
\end{lemma}
\begin{proof}

    Let $\psi\in BL(U)$ with $\Vert\psi\Vert_{BL}\leq 1$. We calculate using that the diameter of the set $U_{\varepsilon,l}$ for all $l=1,...,L$ is at most $\varepsilon$:
    \begin{align*}
        \int\limits_U\psi(x)\mathrm{d}[\mu(0)-\mu_0](x)
        &=\sum_{l=1}^L \int\limits_{U_{\varepsilon,l}}\big(\psi(x)-\psi(z_l)\big)\mathrm{d}[\mu(0)](x)+\int\limits_{U\backslash K} \psi(x)\mathrm{d}[\mu(0)](x)\\
        &\leq \sum_{
        l=1}^L\left(\sup_{x\in U_{\varepsilon,l}}d(x,z_l)\right)[\mu(0)](U_{\varepsilon,l}) +r \Vert \mu(0)\Vert_{TV}\\
        &\leq \varepsilon \Vert {\mu}(0)\Vert_{TV}+r\Vert \mu(0)\Vert_{TV}
        \leq (\varepsilon+r)\Vert \mu(0)\Vert_{TV}.
    \end{align*}
    For the first inequality, we used the tightness of the measure $\mu(0)$ (see \eqref{Eq:tightnessSolution}).
\end{proof}

\section{Proving the convergence order of the numerical scheme}
\label{Sec:convergence}
 In this section, we derive the order of convergence of the numerical scheme proposed in Section \ref{Sec:NumericalScheme}. To this end, we first show that the operator splitting scheme is applicable to the structured population model \eqref{eq:mod_pushforward} and use this to bound the error from the approximating scheme.
 \begin{figure}
    \centering
    \includegraphics[width=\textwidth]{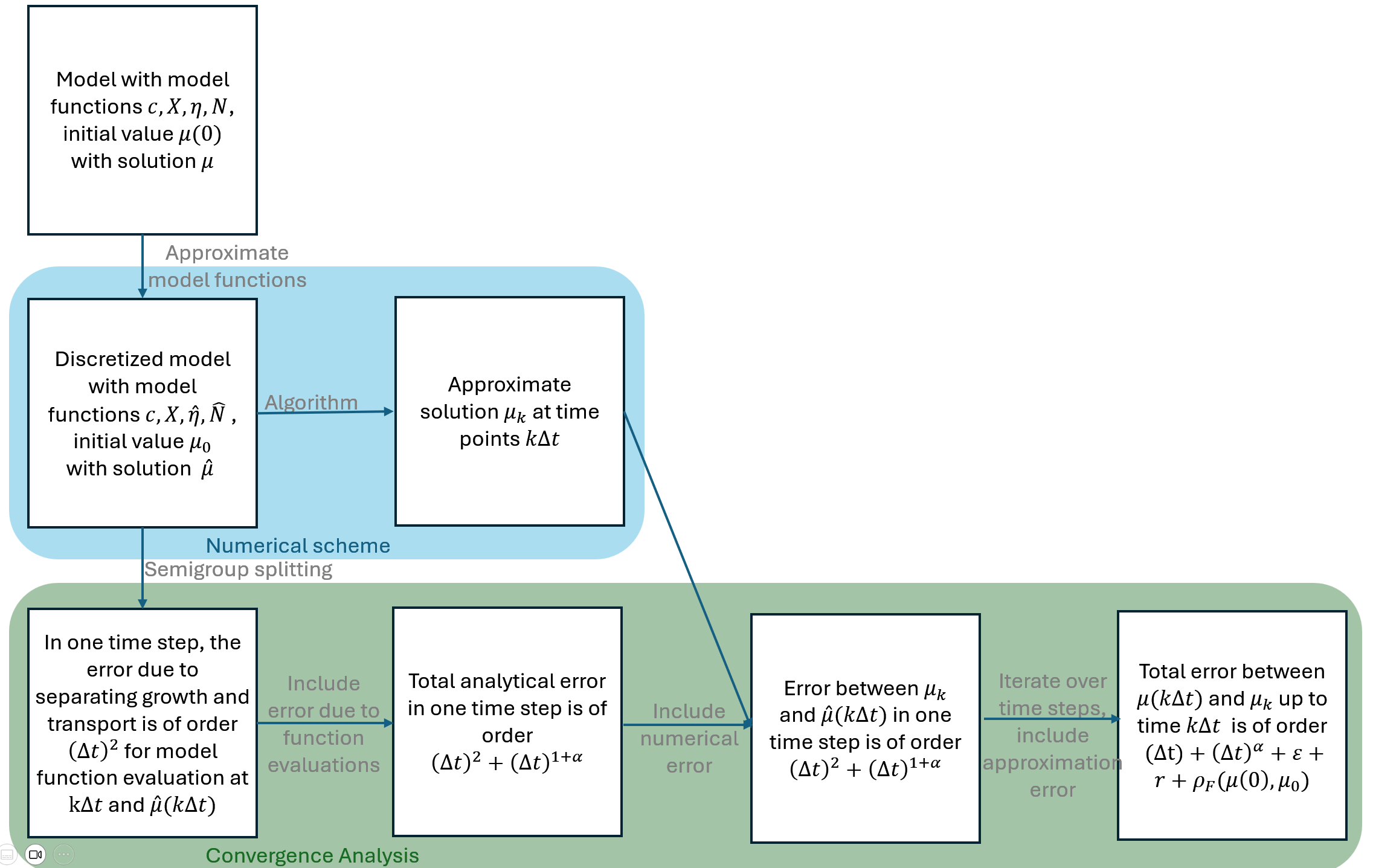}
    \caption{Graphical overview of the convergence analysis conducted in this paper. $\mu$ always refers to a measure, $\Delta t$ is the size of the time steps, $\epsilon$ is the grid size and $r$ controls the cut off error.}
    \label{fig:overview}
\end{figure}
\subsection{Semigroup splitting}
\label{semigroupSplitting}
In the numerical scheme proposed in Section \ref{Sec:NumericalScheme}, the approximation of the solution is split into two steps to separate the effects of the two distinct processes transport and growth. We thus apply the semigroup splitting technique to approximate the solution semigroup by a suitable combination of two semigroups, each corresponding to one of the processes. In this section we follow the framework developed in \cite{MR2050900}.\\

\begin{definition}[Exponentially Lipschitz semigroup \cite{MR2050900}]\ 

Let $\mathcal{T}$ be an operator $\mathcal{T}:[0,\infty)\times U\rightarrow U$ that has the following properties:
\begin{enumerate}
    \item $\mathcal{T}_0=\mathrm{Id}$ and $\mathcal{T}_t\circ\mathcal{T}_s=\mathcal{T}_{t+s}$ for all $t,s\in[0,\infty)$.
    \item For all $T>0,$ $R>0$ and $x_0\in U$ there exists a constant $K=K(T,R,x_0)$ such that for all $t\in[0,T]$ and any $x\in B_R(x_0)$ it holds that $d(\mathcal{T}_tx,x)\leq Kt$.
    \item There exists a constant $C>0$ such that for all $t>0$ and all $x,y\in U$ we have $d(\mathcal{T}_tx,\mathcal{T}_t y)\leq e^{Ct}d(x,y)$.
\end{enumerate}

\end{definition}

Given two semigroups, we want to study the limiting process of applying the semigroups in an alternating manner on small time intervals. Under suitable conditions (as specified in Theorem \ref{operatorSplittingConverges}), the limit coincides with the semigroup generated by both processes together.

\begin{definition}[Semigroup splitting approximation \cite{MR2050900}]\ 

\label{OperatorSplitting Approximation}
  For two exponentially Lipschitz semigroups $\mathcal{T}^1,\mathcal{T}^2$, we define 
\begin{align}
&\Sigma^{2,1,\varepsilon}:[0,\infty)\times U\rightarrow U\\
&\Sigma_t^{2,1,\varepsilon} x= 
\begin{cases}
    \mathcal{T}_t^1 x, & \text{ if } t \in[0,\varepsilon)\\
    \mathcal{T}_{t-n\varepsilon}^1\left(\mathop{\bigcirc}\limits_n \mathcal{T}^2_\varepsilon\mathcal{T}^1_\varepsilon\right)x, &\text{ if } t\in[n\varepsilon, (n+1)\varepsilon) \text{ for some } n\in\mathbb{N},
\end{cases}
\end{align}
where  $\mathop{\bigcirc}\limits_n \mathcal{T}$ denotes the n-times composition of $\mathcal T$.
\end{definition}

\begin{theorem}[Limit of the semigroup splitting approximation\cite{MR2050900}]\ 

    \label{operatorSplittingConverges}
    Let $\mathcal{T}^1,\mathcal{T}^2$ be two exponentially Lipschitz semigroups.
    Assume that for any $T>0$, $R>0$ and $x_0\in U$ there exists a constant $H=H(T,R,x_0)$ such that for any $x\in B_R(x_0)$:
    \begin{align}
    \label{lackOfCommutivity}
        d\big(\mathcal{T}^1_t\mathcal{T}^2_t x,\mathcal{T}^2_t\mathcal{T}^1_t x\big)\leq Ht^2.
    \end{align}
    We refer to this quantity as the lack of commutativity of the two semigroups.\\\
    Under condition \eqref{lackOfCommutivity}, the point-wise limit $\Sigma_t^{2,1,\varepsilon_n} x$ for $n\rightarrow \infty$ exists along the sequence $\varepsilon_n=\varepsilon\cdot 2^{-n}$ for any fixed $\varepsilon>0$. We denote the limit by $\Sigma_t x$.\\
    The limit is uniform in time $t$, i.e. for each $x\in U$
    \begin{align}
        \lim_{n\rightarrow \infty} \sup_{t\in[0,T]} d\left(\Sigma_t^{2,1,\varepsilon_n} x,\Sigma_t x\right)=0.
    \end{align}
    An analogous limiting procedure with the sequence $\Sigma_t^{1,2,\varepsilon_n} x$ yields the same limit $\Sigma_t x$.\\
    
    Further, it holds that \cite[Lemma 2.7]{colombo}
    \begin{align}
        \lim\limits_{t\rightarrow 0} \frac{1}{t^2} d(\mathcal{T}^1_t\mathcal{T}^2_t x,\Sigma_t x)=0,
    \end{align}
    which means that for small time steps $t$, even one iteration of the operator splitting is a good approximation.\\

    Lastly, the limit  $\Sigma_t$ is again an exponentially Lipschitz semigroup with constants $K=K_1+K_2$ and $C=C_1+C_2$.  
\end{theorem}
\begin{remark}
    The condition \eqref{lackOfCommutivity} is used in \cite{MR2050900} instead of a compactness assumption. It is a rather strong assumption but can be fulfilled in our setting if we assume sufficient regularity for the model functions in \eqref{eq:mod_pushforward}.
\end{remark}

In the remainder of this section, we show that the operator splitting method is applicable to the structured population model \eqref{eq:mod_pushforward} and that the limit $\Sigma_t$ is the solution operator to the structured population model. To this end, we divide the model \eqref{eq:mod_pushforward} into two processes, prove that the semigroups generated by these processes are exponentially Lipschitz and fulfill condition \eqref{lackOfCommutivity}. Lastly, we show that the limit of the operator splitting procedure, that exists due to Theorem \ref{operatorSplittingConverges}, is indeed a solution of \eqref{eq:mod_pushforward}.\\

We only need to consider linear and autonomous model functions. This is motivated by the fact, that we want to use this operator splitting method for the convergence of a numerical scheme, in which we evaluate all of the model functions in fixed values at the beginning of each iteration of the scheme (see Section \ref{Sec:NumericalScheme}). \\
For the ease of notation we omit the time and measure arguments of the model functions in this section, i.e. we write for example $c(x)$ instead of $c(t,x,\mu)$.\\

Let $\mathcal{S}^1$ be the one-parameter semigroup on $(\mathcal{M}^+(U),\rho_F)$ given by
\begin{equation}
\label{s1}
    \mathcal{S}_t^1\mu =  X(t,\cdot)_\#\mu \text{ for } t\in[0,T].
\end{equation}
This semigroup is generated by the transport process of model \eqref{eq:mod_pushforward}.\\
Further, let $\mathcal{S}^2$ be the one-parameter semigroup given by
\begin{align}
\begin{split}
    \label{s2}
    \mathcal{S}_t^2\kappa(\cdot)
    =
    \kappa(\cdot)e^{tc(\cdot)}
    &+ \int\limits_0^t \int\limits_U [\hat\eta(x)](\cdot)\mathrm{d}[\mathcal{S}_\tau^2\kappa](x) e^{(t-\tau)c(\cdot)}\mathrm{d}\tau\\
     &+\int\limits_0^t \hat N(\cdot)e^{(t-\tau)c(\cdot)}\mathrm{d}\tau 
    \end{split}
 \end{align}
 for $t\in[0,T]$.\\
Note that in general, the solutions to \eqref{eq:mod_pushforward} form two-parameter semigroups (theorem \ref{Thm:solution}). As the model functions are autonomous in this section, we actually obtain one-parameter semigroups.\\

\begin{lemma}
\label{LipschitzSemigroups}
    The semigroups $\mathcal{S}^1$ and $\mathcal{S}^2$ are exponentially Lipschitz continuous.
\end{lemma}
\begin{proof}
First of all, due to Theorem \ref{Thm:solution}, $\mathcal{S}^1$ and $\mathcal{S}^2$ are indeed semigroups.\\ 

The estimates are obtained from a direct calculation using  the propoerties of the solutions (Theorem \ref{Thm:solution}), the scaling properties of the flat norm (Lemma \ref{scalingFlatNorm}) and Assumptions \ref{assumptionsOnModelFunctions}\eqref{assumptionsSupremum} and \eqref{assumptionsXLipschitzTX}.
\end{proof}

\begin{lemma}[Bounds on the flat norm of the semigroups]\ 

\label{TVSemigroup}
For any $\kappa\in\mathcal{M}^+(U)$ and $t\in[0,T]$, we have 
\begin{align}
    \Vert \mathcal{S}_t^1\kappa\Vert_{BL^*}=\Vert\kappa\Vert_{BL^*}.
\end{align}
Further, for any $R>0$ and $\kappa\in\mathcal{M}^+(U)$ such that $\Vert \kappa \Vert_{BL^*}\leq R$ we have
\begin{align}
    \Vert\mathcal{S}_t^2\kappa\Vert_{BL^*} \leq C 
\end{align}
\end{lemma}
\begin{proof}
    Let $\psi\in BL(U)$ with $\Vert \psi\Vert_{BL}\leq 1$ and let $t\in[0,T]$. We have
    \begin{align*}
        &\int\limits_U \psi(x)\mathrm{d}[\mathcal{S}_t^1\kappa](x)
        = \int\limits_U \psi\big(X(t,x)\big)\mathrm{d}\kappa(x)
        \leq \int\limits_U \mathrm{d}\kappa(x) 
        \leq\Vert \kappa\Vert_{BL^*}
        \end{align*}
         and thus, by taking the supremum over all test functions\\
        \begin{align*}
        & \Vert\mathcal{S}_t^1\kappa\Vert_{BL^*}\leq\Vert \kappa\Vert_{BL^*}    .    
    \end{align*}
    
    \noindent By using the test function $\psi(x) =1$ we obtain
    \begin{align*}
         &\int\limits_U \psi(x)\mathrm{d}[\mathcal{S}_t^1\kappa](x)
         =\int\limits_U \psi\big(X(t,x)\big)\mathrm{d}\kappa(x) =\Vert \kappa\Vert_{TV}.
    \end{align*}
    Combining the above estimates, we obtain the desired equality using the properties of the norms on $\mathcal{M}^+(U)$, see Theorem \ref{propertiesM+}.\\
    For $\mathcal{S}^2$, we can think of $\mathcal{S}_t^2\mu$ as a solution to the model \eqref{eq:mod_pushforward} with the identity as the transport function. As the solutions are narrowly continuous, i.e. continuous w.r.t. the flat norm, and $[0,T]$ is compact, we obtain that
    \begin{align*}
        \sup_{t\in[0,T]} \Vert \kappa(t)\Vert_{BL^*}\leq C.
    \end{align*}
\end{proof}

Now that we established the properties of the solutions to the subproblems, our goal is to show that iteratively solving \eqref{s1} and \eqref{s2} is in fact a good approximation of the exact solution to \eqref{eq:mod_pushforward}. To this end, we need to estimate the lack of commutativity \eqref{lackOfCommutivity} for the semigroups $\mathcal{S}^1$ and $\mathcal{S}^2$.\\
\begin{theorem}
    For any fixed $ \hat \kappa\in\mathcal{M}^+(U)$, $R>0$ and $t\in[0,T]$, it holds that
    \begin{align}
      \rho_F\big(\mathcal{S}^1_t\mathcal{S}^2_t\kappa,\mathcal{S}^2_t\mathcal{S}^1_t\kappa\big)\leq Ct^2 \text{ for all } \kappa\in B_R(\hat\kappa), 
    \end{align}
    where $C=C(R,T,\hat\kappa)$.
\end{theorem}

\begin{proof}
For fixed $\psi\in BL(U)$ with $\Vert \psi\Vert_{BL}\leq 1$ and $t\in[0,T]$ we can write using Lemma \ref{Lemmma:weakform}
\begin{align}
\begin{split}
\label{integrals2s1}
    &\int\limits_U \psi(x)\mathrm{d}[\mathcal{S}_t^2(\mathcal{S}_t^1\kappa)](x)\\
    &= \int\limits_U \psi(x)e^{tc(x)}\mathrm{d}[\mathcal{S}^1_t\kappa](x)
    + \int\limits_0^t \int\limits_U\int\limits_U \psi(y) e^{(t-\tau)c(y)}\mathrm{d}\hat\eta(x)(y)\mathrm{d}[\mathcal{S}^2_\tau\mathcal{S}^1_t\kappa](x)\mathrm{d}\tau
    \\
    &\quad+\int\limits_0^t\int\limits_U\psi(y)e^{(t-\tau)c(y)}\mathrm{d}\hat N(y)\mathrm{d}\tau\\
    &= \int\limits_U \psi\big(X(t,x)\big)e^{tc(X(t,x))}\mathrm{d}\kappa(x)
    + \int\limits_0^t \int\limits_U\int\limits_U \psi(y) e^{(t-\tau)c(y)}\mathrm{d}[\hat\eta(x)](y)\mathrm{d}[\mathcal{S}^2_\tau\mathcal{S}^1_t\kappa](x)\mathrm{d}\tau\\
    &+\int\limits_0^t\int\limits_U\psi(y)e^{(t-\tau)c(y)} \mathrm{d}\hat N(y)\mathrm{d}\tau\\
    &=: A+B+C
    \end{split}
\end{align}
and 
\begin{align}
\begin{split}
\label{integrals1s2}
    &\int\limits_U \psi(x)\mathrm{d}[\mathcal{S}_t^1(\mathcal{S}_t^2 \kappa)](x) 
    = \int\limits_U \psi\big(X(t,x)\big)\mathrm{d}[\mathcal{S}^2_t\kappa](x)\\
    &= \int\limits_U \psi\big(X(t,x)\big)e^{tc(x)}\mathrm{d}\kappa(x)
    + \int\limits_0^t \int\limits_U\int\limits_U \psi\big(X(t,y)\big) e^{(t-\tau)c(y)}\mathrm{d}\hat\eta(x)(y)\mathrm{d}[\mathcal{S}^2_\tau\kappa](x)\mathrm{d}\tau\\
    &\quad+\int\limits_0^t\int\limits_U\psi\big(X(t,y)\big)e^{(t-\tau)c(y)} \mathrm{d}\hat N(y)\mathrm{d}\tau\\
    &=:\tilde{A}+\tilde{B}+\tilde{C}.
    \end{split}
\end{align}
Now, we compare similar terms one by one. For the first term, we use  that $\Vert\psi\Vert_\infty\leq \Vert \psi\Vert_{BL}\leq 1$ and Lemma \ref{boundednessZeta}:
    \begin{align}
    \label{eq:AvsAtilde}
    \begin{split}
        \vert A-\tilde A\vert
       &=\left\vert \int\limits_U \psi\big(X(t,x)\big)\left(e^{t c(X(r,x))}-e^{tc(x)}\right)\mathrm{d}\kappa(x)\right\vert\\ 
        &\leq  \int\limits_U \left\vert e^{tc(X(r,x))}-e^{tc(x)}\right\vert \mathrm{d}\kappa(x)\\
        &\leq e^{TC}\int\limits_U t \left\vert c(X(t,x))-c(x)\right\vert \mathrm{d} \kappa(x)\\
        &\leq  e^{TC}\int\limits_U  t \Vert c\Vert_{BL} d( X(t,x),x) \mathrm{d}\kappa(x)\\
        &\leq C \kappa(U)\cdot t\cdot \omega_X(t) \\
        &\leq C (\hat\kappa(U)+R)\cdot t^2,
        \end{split}
    \end{align}
    where we used the Lipschitz continuity of $c$ and $X$ (Assumption \ref{assumptionsOnModelFunctions}\eqref{assumptionsSupremum} and \eqref{assumptionsXLipschitzTX}) as well as the assumption that $\omega_X(s)\leq Cs$ (Assumption \ref{assumptionsOnModelFunctions}\eqref{assumptionsXLipschitzTX}). We further used that 
    \begin{align}
        \kappa(U)=\Vert \kappa\Vert_{BL^*}\leq \Vert \hat \kappa-\kappa\Vert_{BL^*}+\Vert \hat\kappa\Vert_{BL^*}\leq R+ \Vert \hat \kappa\Vert_{BL^*}.
    \end{align}
    The remaining terms $\vert B-\tilde B\vert$ and  $\vert C-\tilde C\vert$ are estimated similarly to \eqref{eq:AvsAtilde} using Lemmas \ref{LipschitzSemigroups}, \ref{TVSemigroup} and Assumptions \ref{assumptionsOnModelFunctions} \eqref{assumptionsSupremum} and \eqref{assumptionsXLipschitzTX}.
    
    Overall, we obtain
    \begin{align}
        \rho_F\big(\mathcal{S}^1_t\mathcal{S}\kappa,\mathcal{S}^2_t\mathcal{S}^1_t\kappa\big)\leq \vert A-\tilde A\vert+ \vert B-\tilde B\vert +\vert C-\tilde C\vert\leq Ct^2.
    \end{align}
    
\end{proof}

Now that we have shown that the commutator estimate \eqref{lackOfCommutivity} holds for the semigroups $\mathcal S^1$ and $\mathcal S^2$, Theorem \ref{operatorSplittingConverges} yields the following corollary:
\begin{corollary}
\label{semigroupSPMCOnverges}
    The operator splitting scheme applied to the semigroups $\mathcal{S}^1$ and $\mathcal{S}^2$ converges uniformly in time to a limit $\Sigma$, i.e.
    \begin{align}
       \lim_{n\rightarrow \infty} \sup_{t\in[0,T]}  \rho_F\left(\Sigma_t^{2,1,\varepsilon_n} u,\Sigma_t u\right)=0.
    \end{align}
    $\Sigma$ is also an exponentially continuous semigroup.
\end{corollary}

\begin{theorem}
    The limit $\Sigma$ of the semigroup splitting method is a solution operator to the structured population model \eqref{eq:mod_pushforward}, i.e. for $\mu(0)\in \mathcal{M}^+(U)$ the family of measures $\left(\Sigma_t\mu(0)\right)_{t\in[0,T]}$ solves \eqref{eq:mod_pushforward} with model functions $\hat\eta$ and $\hat N$.
\end{theorem}

\begin{proof}

    Let $\hat\mu(\cdot)$ be a solution to \eqref{eq:mod_pushforward} with the model functions $\hat N$ and $\hat\eta$.\\
    First, we will show that for $h<1$, the distance between $\hat \mu(h)$ and $\mathcal{S}_h^1\mathcal{S}_h^2\mu(0)$ in the flat metric is small.\\
    To this end, for a test function $\psi\in BL(U)$  with $\Vert \psi\Vert_{BL}\leq 1$, we bound the difference 
       \begin{align}
       \label{Eq:differenceSplitting}
        \int\limits_U \psi(x)\mathrm{d}\hat\mu_t(x)-\int\limits_U\psi(x)\mathrm d[\Sigma_t\mu(0)](x)
    \end{align}
    using the identities \eqref{eq:Weakform} and \eqref{integrals1s2} for the exact solution and the approximation by the semigroups respectively. By comparing similar terms one by one and using Assmuptions \ref{assumptionsOnModelFunctions} \eqref{assumptionsXLipschitzTX} and \eqref{assumptionsSupremum} as well as Lemma \ref{boundednessZeta} and Theorem \ref{Thm:solution}.
    
    We obtain that 
    \begin{align}
        &\int\limits_U \psi(x)\mathrm{d}[\mathcal{S}_h^1\mathcal{S}_h^2\mu(0)-\hat\mu(h)](x)\leq Ch^2
\end{align}
and thus
\begin{align}
        &\rho_F\big(\mathcal{S}_h^1\mathcal{S}_h^2\mu(0),\hat\mu(h)\big)
        = \sup\limits_{\Vert \psi\Vert_{BL}\leq 1} \int\limits \psi(x) \mathrm{d}[\mathcal{S}_h^1\mathcal{S}_h^2\mu(0)-\hat\mu(h)](x)
        \leq Ch^2.
    \end{align}
As the model functions are time independent, we obtain the same result for solutions starting in $t\in[0,T]$, i.e.
\begin{align}
\label{SemigroupError1step}
    \rho_F\big(\mathcal{S}_h^1\mathcal{S}_h^2\hat\mu(t),\hat\mu({t+h})\big)\leq Ch^2.
\end{align}
Next, we estimate the difference to $\hat\mu$ when we iteratively apply the semigroups $\mathcal{S}^1$ and $\mathcal{S}^2$. \\
We define $\varepsilon_n:=\varepsilon 2^{-n}$ and $t_{n,k}:=\frac{kT}{2^n}$ for $k=0,..,2^n-1$. Thus, we obtain for $\varepsilon_n$ sufficiently small:
\begin{align*}
    &\rho_F\big(\hat\mu({t_{n,k}}),\left(\mathcal{S}_{\varepsilon_n}^1\mathcal{S}_{\varepsilon_n}^2\right)^k\mu(0)\big)
    =\rho_F\big(\hat\mu({t_{n,k}}),\mathcal{S}_{\varepsilon_n}^1\mathcal{S}_{\varepsilon_n}^2\left(\mathcal{S}_{\varepsilon_n}^1\mathcal{S}_{\varepsilon_n}^2\right)^{k-1}\mu(0)\big)\\
    &\leq \rho_F\big(\hat\mu({t_{n,k}}),\mathcal{S}_{\varepsilon_n}^1\mathcal{S}_{\varepsilon_n}^2\hat\mu({t_{n,k-1}})\big)
    +\rho_F\big(\mathcal{S}_{\varepsilon_n}^1\mathcal{S}_{\varepsilon_n}^2\hat\mu({t_{n,k-1}}),\mathcal{S}_{\varepsilon_n}^1\mathcal{S}_{\varepsilon_n}^2\left(\mathcal{S}_{\varepsilon_k}^1\mathcal{S}_{\varepsilon_n}^2\right)^{k-1}\mu(0)\big)\\
    &\leq C(\varepsilon_n)^2+e^{(C_1+C_2)\varepsilon_n}\rho_F\big(\hat\mu({t_{n,k-1}}),\left(\mathcal{S}_{\varepsilon_n}^1\mathcal{S}_{\varepsilon_n}^2\right)^{k-1}\mu(0)\big),
\end{align*}
where we used estimate \eqref{SemigroupError1step} and the exponential Lipschitz regularity of the semigroups (Lemma \ref{LipschitzSemigroups}).\\
With an iteration argument (Lemma \ref{iterationLemma}) and using that at time $0$ the solutions coincide, i.e. $\mathcal{S}_0^1\mathcal{S}_0^2\mu(0)=\mu(0)$, we thus conclude that:
\begin{align}
\label{estimateOnInterval}
     &\rho_F\big(\hat\mu({t_{n,k}}),\left(\mathcal{S}_{\varepsilon_n}^1\mathcal{S}_{\varepsilon_n}^2\right)^k\mu(0)\big)
     \leq  C\frac{e^{k(C_1+C_2)\varepsilon_n}-1}{e^{(C_1+C_2)\varepsilon_n}-1}(\varepsilon_n)^2.
\end{align}
Note, that by definition, the identity $\Sigma_{t_{n,k}}^{(k)}=\left(\mathcal{S}_{\varepsilon_n}^1\mathcal{S}_{\varepsilon_n}^2\right)^k$ holds for all time points $t_{n,k}$.\\
For an arbitrary $t\in[0,T]$ we can use the Lipschitz continuity of the solutions to \eqref{eq:mod_pushforward} with respect to time (Theorem \ref{Thm:solution}). Let $k$ be such that $t\in[t_{n,k},t_{n,k+1}]$.
\begin{align}
\label{Eq:SplittingEstimateInterval}
    &\rho_F\big(\hat\mu(t),\Sigma_t^{(n)}\mu(0)\big) \\
    &\leq \rho_F\bigg(\hat\mu(t),\hat\mu({t_{n,k}})\bigg)+\rho_F\bigg(\hat\mu({t_{n,k}}),\Sigma_{t_{n,k}}^{(n)}\mu(0)\bigg)+\rho_F\big(\Sigma_{t_{n,k}}^{(n)}\mu(0),\Sigma_t^{(n)}\mu(0)\big)\\
    &\leq C\vert t-t_{n,k}\vert + C\frac{e^{k(C_1+C_2)\varepsilon_n}-1}{e^{(C_1+C_2)\varepsilon_n}-1}(\varepsilon_n)^2 +(K_1+K_2)\vert t-t_{n,k}\vert .
\end{align}
Here, we used \eqref{estimateOnInterval} and \cite[Proposition 3.2]{MR2050900}, which states that $\Sigma^{(n)}$ is Lipschitz continuous in time with constant $K_1+K_2$ where $K_1$ and $K_2$ are the Lipschitz constants of the semigroups $S^1$ and $S^2$ respectively.\\
As we know that $\sum\nolimits_t^{(n)}\mu(0)$ converges to $\sum\nolimits_t\mu(0)$ uniformly (see Theorem \ref{operatorSplittingConverges}), we obtain with \eqref{Eq:SplittingEstimateInterval}:
\begin{align*}
    &\rho_F\big(\hat\mu(t),\Sigma_t\mu(0)\big) 
    \leq \rho_F\big(\hat\mu(t),\Sigma_t^{(n)}\mu(0)\big) +\rho_F\big(\Sigma_t^{(n)}\mu(0),\Sigma_t\mu(0)\big) 
    \rightarrow 0\text{ as } n\rightarrow \infty.
\end{align*}
Hence,
\begin{align}
    \hat\mu(t)=\Sigma_t\mu(0) 
\end{align}
for all $t\in[0,T]$.
\end{proof}

\subsection{Convergence Analysis}
\label{convergence}

In this section, we analyze the error from the algorithm proposed in Section \ref{Sec:NumericalScheme}. We study the error that arises from splitting the problem instead of solving the transport and the growth/jump problems at the same time.\\

\begin{remark}
    As Theorem \ref{Theorem:finiteRangeApproximation} only provides an approximation of the model functions $\eta$ and $N$ in the linear case, we need to clarify the notation a bit before proceeding. Throughout this section we will write $\hat\eta(t,x,\mu)$ and $\hat N(t,\mu)$ in order to avoid having to different between the two cases. If $\eta$ and $N$ are linear, so are $\hat\eta$ and $\hat N$, and the dependence on $\mu$ is artificial. In the nonlinear case, we set $\hat\eta=\eta$ and $\hat N=N$. In both cases, some of the terms in the proof will vanish.
\end{remark}

For the analysis, we consider a semi-continuous auxiliary scheme: 
Given the value $\mu_k$ at the beginning of the time step starting at time $k\Delta t$, we first solve an equation that includes only the transport process:
 \begin{align}
 \label{justTransport1}
     \bar\mu(t) &= \bar{X}(t,\cdot)_\#\mu_k(\cdot) \text{ in } [k\Delta t,(k+1)\Delta t],
 \end{align}
where
\begin{align}
    \label{transportAuxiliary}
        &\bar{X}(t,\cdot)=X(t,k\Delta t,\cdot,{\mu}_{k}).
    \end{align}
We denote by $\bar{\mu}((k+1)\Delta t)$ the solution to \eqref{justTransport1} at time $(k+1)\Delta t$, which is used in the second equation.
Equation \eqref{justTransport1} describes the transport process and corresponds to the semigroup $\mathcal{S}^1\mu_k$ from Section \ref{semigroupSplitting}.\\
The second equation represents non-local interactions, influx and growth:
\begin{align}
     \label{grwothNonlocal1}
     \begin{split}
     \Breve{\mu}(t)
     &=\bar{\mu}((k+1)\Delta t)\exp\left(\int\limits_{\ k\Delta t}^t\Breve c(\cdot)\mathrm{d}s\right)
     \\
     &+\int\limits_{k\Delta t}^t \int\limits_U [\Breve{ \eta}(x)(\cdot)]\mathrm{d}\Breve\mu(\tau)(x) \exp\left(\int\limits_{\tau}^t\Breve{c}(\cdot)\mathrm{d}s\right)\mathrm{d}\tau
    + \int\limits_{k\Delta t}^t\Breve{N} (\cdot)\exp\left(\int\limits_{\tau}^t\Breve c(\cdot)\mathrm{d}s\right)\mathrm{d}\tau   .     
    \end{split}
 \end{align}
 This corresponds to applying the operator $\mathcal{S}^2$ from Section \ref{semigroupSplitting} to $\bar\mu((k+1)\Delta t)$.\\ 
 Here, we use the following functions:
 \begin{align*}
     &\Breve c(x) = c(k\Delta t,x,\bar{\mu}((k+1)\Delta t)),\\
        & \Breve{\eta}(x)(\cdot)= \hat\eta(k\Delta t, x,\bar\mu((k+1)\Delta t)),\\
        & \Breve N(\cdot) =  \hat N(k\Delta t,\bar\mu((k+1)\Delta t)).
 \end{align*}
 Note that these auxiliary models are linear and autonomous as we evaluate the model functions at fixed measures and time points.\\
\begin{table}
\label{Table:AuxSchemes}
\caption{Overview of the auxiliary schemes used in the proof of Lemma \ref{ErrorAlgorithm}. $\bar \mu$ and $\Breve\mu$ are used for the semigroup splitting, $\nu$ is the limit of this semigroup splitting procedure. $\lambda$ is used as a link between $\nu$ and the exact solution $\hat \mu$ by using the same model functions as $\nu$ but the inital value from $\hat\mu$.}
\begin{tabular}{ |c|c|c| } 
 \hline
Solution & Initial value & Model functions \\ 
\hline\hline & \\[-1.5ex]
 $\hat\mu$ & $\hat\mu(k\Delta t)$ & \begin{tabular}{@{}c@{}}$X(t,\tau,\cdot,\hat\mu(\cdot))$, $c(t,x,\hat\mu(t))$,\\ $\hat\eta(t,x,\hat\mu(t))(\cdot)$, $\hat N(t,\hat\mu(t))(\cdot)$\end{tabular}   \\ 
 \hline& \\[-1.5ex]
 $\bar\mu$ & $\mu_k$ & $\bar X(t,\cdot)=X(t,k\Delta t,\cdot,\mu_k)$ \\ 
 \hline & \\[-1.5ex]
 $\Breve{\mu}$ & $\bar\mu((k+1)\Delta t)$ & \begin{tabular}{@{}c@{}} $\Breve c(x) = c(k\Delta t,x,\bar{\mu}((k+1)\Delta t))$, $ \Breve{\eta}(x)(\cdot)= \hat\eta(k\Delta t, x,\bar\mu((k+1)\Delta t))(\cdot)$,\\ $\Breve N(\cdot) =  \hat N(k\Delta t,\bar\mu((k+1)\Delta t))(\cdot)$ \end{tabular}\\
 \hline& \\[-1.5ex]
 $\nu$ & $\mu_k$ & \begin{tabular}{@{}c@{}}$\bar{X}(t,\cdot)=X(t,k\Delta t,\cdot,{\mu}_{k})$, $\bar{c}(x) = c(k\Delta t,x,{\mu}_{k})$,\\ $ \bar{\eta}(x)(\cdot)= \hat \eta(k\Delta t,x,\mu_k)(\cdot)$, $ \bar{N}(\cdot) = \hat N(k\Delta t,\mu_k)(\cdot)$\end{tabular}\\
 \hline& \\[-1.5ex]
 $\lambda$& $\mu(k\Delta t)$ & $\bar X$, $\bar c$, $\bar \eta$, $\bar N$\\
 \hline
\end{tabular}
\end{table}
\begin{lemma}
    [Splitting error in one time step]\ 
    
\label{ErrorAlgorithm}
\noindent Denote by $\hat\mu(k\Delta t)$ the exact solution to \eqref{eq:mod_pushforward} at time $t\in [0,T]$ using the approximative model functions $\hat\eta $ and $\hat N$ obtained from Theorem \ref{Theorem:finiteRangeApproximation}. Further denote by $\Breve \mu_{k\Delta t}$ the solution as obtained by the auxiliary scheme given by \eqref{justTransport1} and \eqref{grwothNonlocal1}. Then, the splitting error is bounded by
\begin{align}
\label{errorBetweenReconstructions}
    \rho_F\big(\hat\mu(k\Delta t),\Breve \mu(k\Delta t)\big)
    &\leq C(e^{C\Delta t})\rho_F\big(\hat\mu((k-1)\Delta t),\Breve \mu((k-1)\Delta t)\big)+ C\left((\Delta t)^2+(\Delta t)^{1+\alpha})\right) .
\end{align}

\end{lemma}

\begin{proof}
In order to use the semigroup splitting approach described in Section \ref{semigroupSplitting}, we introduce yet another auxiliary problem, which corresponds to the limit of the semigroup splitting scheme applied to $\mathcal{S}^1$ and $\mathcal{S}^2$. In particular, let $\nu$ be the exact solution of \eqref{eq:mod_pushforward} with initial value $\mu_k$ and  model functions 
\begin{align}
        &\bar{X}(t,\cdot)=X(t,k\Delta t,\cdot,{\mu}_{k}),\\
        &\bar{c}(x) = c(k\Delta t,x,{\mu}_{k}),\\
        & \bar{\eta}(x)(\cdot)= \hat \eta(k\Delta t,x,\mu_k)(\cdot),\\
        & \bar{N}(\cdot) = \hat N(k\Delta t,\mu_k)(\cdot).
 \end{align}
We use the auxiliary schemes to obtain solutions that we can easily compare. $\Breve \mu$ and $\nu$ can be compared via the operator splitting process, see Section \ref{semigroupSplitting}, whereas $\nu$ and $\hat \mu$ can be compared using the continuity of the solution w.r.t. parameter functions and initial data, see Theorem \ref{Thm:solution}. An overview of all of the solutions that we compare in this proof and of the model functions that are used for them is given in Table \ref{Table:AuxSchemes}.
    
By Theorems \ref{operatorSplittingConverges}  and \ref{semigroupSPMCOnverges}, we obtain that 
\begin{equation}
    \rho_F\big(\Breve{\mu}((k+1)\Delta t),\nu((k+1)\Delta t)\big)\leq C(\Delta t)^2.
\end{equation}
The difference $\rho_F\big(\hat \mu((k+1)\Delta t),\nu((k+1)\Delta t)\big)$ can be bound by using the continuous dependence of solutions on the model functions and initial values (Theorem \ref{Thm:solution}).  
To this end, let $\lambda(\cdot)\in C^0([0,T],\mathcal{M}^+(U))$ be a solution to \eqref{eq:mod_pushforward} on $[k\Delta t,(k+1)\Delta t]$ with exact initial value $\mu(k\Delta t)$ but model functions $\bar X, \bar \eta,\bar N$ and $\bar c$. \\
It holds that
\begin{align}
\label{errorlambda}
    &\rho_F\big(\hat\mu((k+1)\Delta t),\nu((k+1)\Delta t)\big)\\
    &\leq \rho_F\big(\hat\mu((k+1)\Delta t),\lambda((k+1)\Delta t)\big)+\rho_F\big(\lambda((k+1)\Delta t),\nu((k+1)\Delta t)\big)
\end{align}
so that we can estimate the terms one by one.
The solutions $\lambda$ and $\nu$ only differ in the initial value, such that for the second term in \eqref{errorlambda} we obtain by Theorem \ref{Thm:solution}:
\begin{align}
\label{errorlambda1}
\begin{split}
    \rho_F\big(\lambda((k+1)\Delta t),\nu((k+1)\Delta t)\big)
    &\leq C \rho_F\big(\lambda(k\Delta t),\nu(k\Delta t)\big) \exp\left(C\Delta t \sup_{x\in U} \Vert \bar\eta(x)\Vert_{BL^*}\right)\\
    &= C \rho_F\big(\mu_k,\mu(k\Delta t)\big) \exp\left(C\Delta t \sup_{x\in U} \Vert \bar \eta(x)\Vert_{BL^*}\right).
    \end{split}
\end{align}
For the first term in \eqref{errorlambda}, in which there is a difference in the model functions (see Table  \ref{Table:AuxSchemes}), we use the Lipschitz continuity w.r.t. model functions (Theorem \ref{Thm:solution}):
\begin{align}
\label{errorlambda2}
\begin{split}
    &\rho_F\big(\hat\mu((k+1)\Delta t),\lambda((k+1)\Delta t)\big)\\
    &\leq 
    C\int\limits_{k\Delta t}^{(k+1)\Delta t} \sup_{x\in U}  \rho_F\big(\hat\eta(\tau,x,\hat\mu(\tau)), \hat\eta(k\Delta t,x,\mu_k)\big) \mathrm{d}\tau\\
    &\quad+C \int\limits_{k\Delta t}^{(k+1)\Delta t} \sup_{t\in[k\Delta t,(k+1)\Delta t]} \rho_F\big(\hat N(\tau,\hat\mu(\tau)),\hat N(k\Delta t,\mu_k)\big)\mathrm{d}\tau\\
    &\quad+ C\int\limits_{k\Delta t}^{(k+1)\Delta t}\Vert c(\tau,\cdot,\hat\mu(\tau))-c(k\Delta t,\cdot,\mu_k))\Vert_\infty \mathrm{d}\tau\\
    &\quad+C\sup_{k\Delta t\leq\tau_1\leq \tau_2\leq (k+1)\Delta t}\sup_{x\in U} d\big(X(\tau_2,\tau_1,\cdot,\mu(\cdot)),X(\tau_2, k\Delta t,\cdot,\mu_k)\big)\\
    &:= I_1+I_2+I_3+I_4\\
   &\leq C\Delta t \big(\rho_F\big(\hat\mu(k\Delta t),\mu_k\big)\big)+C(\Delta t)^2+C(\Delta t)^{{1+\alpha}},
   \end{split}
\end{align}
where we obtain the final error estimate by bounding the terms separately and assuming that $\Delta t\leq 1$ which will be done in detail below.\\
Here, we need the $C^{0,\alpha}$ regularity of the model functions with respect to time as stated in Assumptions \ref{assumptionsOnModelFunctions}\eqref{assumptionsOnModelFunctionsTime}.\\
Before we bound $I_1$ to $I_4$, note that
\begin{align}
\label{estiamteMuLipschitz}
    &\rho_F\big(\hat\mu(t),\mu_k\big)
    \leq \rho_F(\hat\mu(t),\hat\mu(k\Delta t))+\rho_F(\hat\mu(k\Delta t),\mu_k)
    \leq C\Delta t +\rho_F(\hat\mu(k\Delta t),\mu_k)
\end{align}
by the Lipschitz continuity of solutions (Theorem \ref{Thm:solution}).\\
We now estimate the terms $I_1$ to $I_4$ separately.
\begin{align*}
I_1
&\leq C \Delta t \sup_{t\in[k\Delta t,(k+1)\Delta t]}\sup_{x\in U} \left(\rho_F\big(\hat\eta(t,x,\hat\mu(t)),\hat\eta(t,x,\mu_k)\big)+\rho_F\big(\hat\eta(t,x,\mu_k),\hat\eta(k\Delta t,x,\mu_k)\big)\right)\\
& \leq C \Delta t \sup_{t\in[k\Delta t,(k+1)\Delta t]} \left(L_{\hat\eta,R}\rho_F\big(\hat\mu(t),\mu_k\big)+\vert t-k\Delta t\vert^\alpha \Vert \hat\eta\Vert_{C^{0,\alpha}}\right)\\
&  \leq C \Delta t \left( \sup_{t\in[k\Delta t,(k+1)\Delta t]}L_{\hat\eta,R}(\rho_F\big(\hat\mu(t_k),\mu_k\big)+C\Delta t)+(\Delta t)^\alpha \Vert \eta\Vert_{C^{0,\alpha}}\right).
\end{align*}

Here, $R=\sup_{t\in[0,T]}\Vert \hat\mu(t)\Vert_{TV}$ which is bounded due to Theorem \ref{Thm:solution}.\\

Similarly, we obtain
\begin{align*}
    I_2
    &\leq C \Delta t  \sup_{t\in[k\Delta t,(k+1)\Delta t]}\left(\rho_F\big(\hat N(t,\hat\mu(t)),\hat N(t,\mu_k)\big)+\rho_F\big(\hat N(t,\mu_k),\hat N(k\Delta t,\mu_k)\big)\right)\\
    &\leq  C \Delta t \sup_{t\in[k\Delta t,(k+1)\Delta t]}\left( L_{\hat N,R}\big(\rho_F\big(\hat \mu(k\Delta t),\mu_k\big)+C\Delta t\big)+\vert t-k\Delta t\vert^\alpha\Vert \hat N\Vert_{C^{0,\alpha}}\right)\\
    &\leq C \Delta t \left( \sup_{t\in[k\Delta t,(k+1)\Delta t]} L_{\hat N,R}\big(\rho_F\big(\hat\mu(k\Delta t),\mu_k\big)+C\Delta t\big)+(\Delta t)^\alpha\Vert \hat N\Vert_{C^{0,\alpha}}\right).\\
\end{align*}
The third term can be estimated by
\begin{align*}
 I_3
 &\leq C \Delta t  \sup_{t\in[k\Delta t,(k+1)\Delta t]}\left( \Vert c(t,\cdot,\hat\mu(t))-c(t,\cdot,\mu_k)\Vert_\infty + \Vert c(t,\cdot,\mu_k)-c(k\Delta t,\cdot,\mu_k)\Vert_\infty \right)
\\
&\leq C \Delta t  \sup_{t\in[k\Delta t,(k+1)\Delta t]} \left(L_{C,R}\left(\rho_F\big(\hat\mu(k\Delta t),\mu_k\big)+C\Delta t\right) +\vert t-k\Delta t\vert^\alpha \Vert c\Vert_{C^{0,\alpha}}\right)\\
&\leq C \Delta t \left( \sup_{t\in[k\Delta t,(k+1)\Delta t]} L_{C,R}\left(\rho_F\big(\hat\mu(k\Delta t),\mu_k\big)+C\Delta t\right) +(\Delta t)^\alpha \Vert c\Vert_{C^{0,\alpha}}\right).
\end{align*}

For the term $I_4$, we use the semigroup property of $X$ (Assumptions \ref{assumptionsOnModelFunctions} \eqref{AssumptionSemigroup}) which yields
\begin{align}
    X(\tau_2,t_k,x,\mu_k)
    =X(\tau_2,\tau_1,X(\tau_1,t_k,x,\mu_k),\mu_k)
\end{align} 
for any $k\Delta t\leq \tau_1\leq \tau_2\leq (k+1)\Delta t$.\\
\begingroup
\allowdisplaybreaks
Hence, we obtain 
\begin{align}
    I_4
    &= \sup_{k\Delta t\leq \tau_1\leq\tau_2\leq (k+1)\Delta t}\ \sup_{x\in U} d\big(X(\tau_2,\tau_1,x,\hat\mu(\cdot)), X(\tau_2,\tau_1,X(\tau_1,t_k,x,\mu_k),\mu_k)\big)\\
    &\leq \sup_{k\Delta t\leq \tau_1\leq\tau_2\leq (k+1)\Delta t}\ \sup_{x\in U} \big( d\big(X(\tau_2,\tau_1,x,\hat\mu(\cdot)),X(\tau_2,\tau_1,x,\mu_k)\big) \\
    &+ d\big(X(\tau_2,\tau_1,x,\mu_k),X(\tau_2,\tau_1, X(\tau_1,k\Delta t,x,\mu_k), \mu_k)\big) \big)\\
    &\leq \sup_{k\Delta t\leq \tau_1\leq\tau_2\leq (k+1)\Delta t}\ \sup_{x\in U}\left(\int\limits_{\tau_1}^{\tau_2}L_{X,R}(s)\rho_F\big(\hat\mu(s),\mu_k\big)\mathrm{d}s 
    + L_X(\tau_2-\tau_1) d\big(X(\tau_1,k\Delta t,x,\mu_k),x\big)
    \right)\\
    &\leq \sup_{k\Delta t\leq \tau_1\leq\tau_2\leq (k+1)\Delta t}\ \sup_{x\in U}\left(\int\limits_{\tau_1}^{\tau_2}L_{X,R}(s)\big(C\Delta t+\rho_F\big(\hat\mu(k\Delta t,\mu_k)\big)\big)\mathrm{d}s 
    + L_X\big(\tau_2-\tau_1\big) \omega\big(\tau_1-k\Delta t\big)
    \right)\\
    &\leq \sup_{k\Delta t\leq \tau_1\leq\tau_2\leq (k+1)\Delta t}\ \sup_{x\in U}\left( \int\limits_{\tau_1}^{\tau_2}L_{X,R}(s)\big(C\Delta t+\rho_F\big(\hat\mu(k\Delta t,\mu_k)\big)\big)\mathrm{d}s 
    + e^{C(\tau_2-\tau_1)} C(\tau_1-k\Delta t)
    \right)\\
    &\leq C\left((\Delta t)^2 +\Delta t \rho_F\big((\hat\mu(k\Delta t,\mu_k)\big) +e^{C\Delta t}\Delta t\right) 
    \\
    &\leq C (\Delta t)^2 +C \Delta t \rho_F\big((\hat\mu(k\Delta t,\mu_k)\big),
\end{align}
\endgroup
where we used the boundedness of $L_{X,R}$, the Lipschitz continuity of the solutions (Theorem \ref{Thm:solution}), $\omega_X(t)\leq Ct$ as well as $L_R(t)\leq e^{Ct}$ (Assumptions \ref{assumptionsOnModelFunctions} \eqref{assumptionsXLipschitzTX}).\\

By plugging \eqref{errorlambda1} and \eqref{errorlambda2} into \eqref{errorlambda}, we obtain:
\begin{align}
\label{error1step}
\begin{split}
    \rho_F\big(\hat\mu((k+1)\Delta t),\Breve{\mu}((k+1)\Delta t)\big)
    &\leq C(\Delta t+e^{c\Delta t})\rho_F\big(\hat\mu(k\Delta t),\mu_k\big)+C\big((\Delta t)^2+(\Delta t)^{1+\alpha}\big)\\
    &\leq C(e^{c\Delta t})\rho_F\big(\hat\mu(k\Delta t),\mu_k\big)+C\big((\Delta t)^2+(\Delta t)^{1+\alpha}\big)\\
    \end{split}
\end{align}
as $\Delta t +e^{C\Delta t}\leq C e^{C\Delta t}$ for small $\Delta t$.\\

\end{proof}

\begin{lemma}
\label{errorBetweenRecNumerical}
Let $k\Delta t, m\Delta t\in [0,T]$, with $k>m$.
    Denote by $\hat\mu(k\Delta t)$ the exact solution to \eqref{eq:mod_pushforward} using the approximate model functions $\hat\eta$ and $\hat N$ and denote by $\mu_k$ the result of the numerical scheme from Section \ref{Sec:NumericalScheme} including numerical errors. The error between those two measures is bounded by
\begin{align}
    \rho_F\big(\hat\mu(k\Delta t),\mu_k\big)
    &\leq C(e^{C\Delta t})^{(k-m)}\rho_F\big(\hat\mu(m\Delta t),\mu_m\big)+ C(k-m)\big((\Delta t)^2+(\Delta t)^{1+\alpha})\big) .
\end{align}
\end{lemma}
\begin{proof}
The error estimates from Theorem \ref{ErrorAlgorithm} were made using exact solutions to Equation \eqref{grwothNonlocal1} describing the evolution of the masses in the points. In the numerical scheme, we use an approximation by an explicit Euler scheme instead (see \eqref{MassEvolution}). Hence, we also need to consider the numerical errors.\\
One step of the explicit Euler scheme has an error of order $(\Delta t)^2$ \cite{ODESolver} if the differential operator is continuously differentiable. \\
As we require that $c,\beta_l,n_l\in C^1([0,T])$ (Assumption \ref{assumptionsOnModelFunctions}\eqref{C1assumption} or \eqref{Ass:finiteRangeNonlinear}), this applies to \eqref{MassEvolution}.\\
Further, evaluating the transport function in $k\Delta t$ as in \eqref{trasnportNumericalScheme]} rather than in $t\in[k\Delta t,(k+1)\Delta t]$ as in \eqref{transportAuxiliary} leads to an error of order $(\Delta t)^{1+\alpha}$ due to the regularity of $X$ in time (see Assumption \ref{assumptionsOnModelFunctions}\eqref{assumptionsOnModelFunctionsTime}).\\
Hence, estimate \eqref{error1step} also holds true if we replace $\Breve\mu(m\Delta t)$ with $\mu_m$.
    We use this estimate iteratively by employing Lemma \ref{iterationLemma}:
\begin{align*}
    &\rho_F\big(\hat\mu(k\Delta t),\mu_k\big)\\
    &\leq C(e^{C\Delta t})^{(k-m)}\rho_F\big(\hat\mu(m\Delta t),\mu_m\big)+\frac{ C(e^{C\Delta t})^{(k-m)}-1}{ Ce^{C\Delta t}-1} C\big((\Delta t)^2+(\Delta t)^{1+\alpha})\big)  \\
    &\leq C(e^{C\Delta t})^{(k-m)}\rho_F\big(\hat\mu(m\Delta t),\mu_m\big)+ C(k-m)\big((\Delta t)^2+(\Delta t)^{1+\alpha}\big) .
\end{align*} 
Here we used that $e^x-1\geq x$ for all $x\in\mathbb{R}$ and that on a compact interval, $e^x-1\leq Cx$. Thus for $(k-l)\Delta t\in[0,T]$
\begin{align}
    \label{fraction}
    \frac{ C(e^{C\Delta t})^{(k-m)}-1}{ Ce^{C\Delta t}-1}
    \leq C \frac{\Delta t (k-m)}{\Delta t}\leq C(k-m).
\end{align}

\end{proof}

\begin{corollary}

 Let $k \Delta t\in [0,T]$. Assume that $\Delta t\leq 1$.
    The overall error of the numerical scheme can be estimated by
    \begin{align}
        &\rho_F\big(\mu(k\Delta t),\mu_k\big)
    \leq C\bigg(r+\varepsilon + \rho_F\big(\mu(0),\mu_0\big)+[(\Delta t)+(\Delta t)^\alpha]\bigg).
    \end{align}
    The constant $C$ is allowed to depend on the model functions and on $T$. 
\end{corollary}
\begin{proof}
We can split the error into an approximation error and a time-discretization error:
\begin{align}
    \rho_F\big(\mu(k\Delta t),\mu_k\big)
    \leq 
    \rho_F\big(\mu(k\Delta t),\hat\mu(k\Delta t)\big)+\rho_F\big(\hat\mu(k\Delta t),\mu_k\big).
\end{align}
In the case that $\eta$ and $N$ are nonlinear, $\mu=\hat\mu$ and the first term vanishes. If $\eta$ and $N$ are linear, we obtain by the continuous dependence of the solution on the model functions (Theorem \ref{Thm:solution}):
\begin{align}
    &\rho_F\big(\mu(k\Delta t),\hat\mu(k\Delta t)\big)\\
    &\leq C\int\limits_0^{k\Delta t}\sup_{x\in U}\rho_F\big(\eta(\tau,x),\hat\eta(\tau,x)\big)+\rho_F\big(N(\tau),\hat N(\tau)\big)\mathrm{d}\tau\\
    &\leq C\varepsilon\int\limits_0^{T} \sup_{x\in U} \Vert \eta(\tau,x)\Vert_{BL^*}+\Vert N(\tau)\Vert_{BL^*}\mathrm{d}\tau\\
    &\leq C(\varepsilon +r)
\end{align}
by Theorem \ref{Theorem:finiteRangeApproximation}.
Further, we obtain using Lemma \ref{errorBetweenRecNumerical} with $m=0$ and using that the number of time steps in the algorithm is at most $\frac{T}{\Delta t}$, i.e. $k\leq\frac{T}{\Delta t}$
\begin{align}
    \rho_F\big(\hat\mu(k\Delta t),\mu_k\big)
    &\leq C(e^{C\Delta t})^{k}\rho_F\big(\hat\mu(0),\mu_0\big)+ Ck\bigg((\Delta t)^2+(\Delta t)^{1+\alpha}\bigg)\\    
    &\leq C\bigg( \rho_F\big(\hat\mu(0),\mu_0\big)+(\Delta t)+(\Delta t)^\alpha\bigg).
\end{align}
\end{proof}

\section{Discussion}
High-dimensional experimental data for the structure of populations enforces consideration of structured population models beyond the classical euclidean setting. An examples of this are the dynamics of heterogeneous cell system, which are nowadays characterized by time series of single-cell sequencing data. 

Structured population models on the space of Radon measures are a useful tool to study the evolution of populations that are structured by an underlying variable that is best described by a non-Euclidean state space. We are able to use a state space which is merely a separable and complete metric space.\\
The minimal structure and generality of this modeling framework made it necessary to develop a suitable numerical method for simulations.
To this end, we proposed a particle based operator-splitting algorithm which is of linear accuracy in the mesh density and the spatial cut-off parameter. The accuracy in time is 
$(\Delta t)^\alpha$, where $\Delta t$ is the size of the time steps and the model functions are $\alpha$ Hölder regular in time. \\
This allows to simulate a broad class of structured population models with discrete and continuous transitions, an influx function and especially with non-constant mass. \\
Further, the numerical solutions can be used to approximate the posterior density in a Bayesian inverse setting, allowing us to connect observations of the solution, such as moments of the measure, to a finite dimensional parametrization of the model.\\
The flexibility of the particle method proposed in this paper allows to place a higher density of mesh points in areas of particular interest for the application, as the only requirement on the grid is, that it is at least $\varepsilon$ dense.

The choice of the grid hat we use to approximate the measure is however also a main limitation of our approach. We assume that we know the countable dense subset of the separable metric space. While for many spaces, this set is either well known or can be constructed easily, it is not obvious how to construct such a countable dense subset for any given separable metric space.

The number of points needed to cover a set up to a given accuracy and hence the number of equations needed to attain a certain prescribed accuracy may explode.
The curse of dimensionality is a common phenomenon for many numerical methods for solving differential equations. In numerical integration, this problem is treated by (quasi) Monte Carlo methods which provide dimension independent convergence rates that depend only on the number of sample points used to calculate the integral \cite{Novak1988-ih}, \cite{Sloan1998}.\\
In our setting, we might overcome this dimension-dependence issue by sampling the measure for the initial approximation and the approximations of the model functions $N$ and $\eta$, instead of using the deterministic approximation that we proposed in this paper. \\
\\
The strong assumption needed for model functions $N$ and $\eta$ are another limitation of our approach. For our approximation, we needed to assume that the functions are linear. While this is an important constraint, many applications such as $\eta$ representing diffusion or $N$ representing a constant influx from an external source are already included. It would also be possible to assume a priori that these model functions have a finite image as it was done in \cite{Carillo}. In this case, the model functions can be nonlinear but this a strong assumption on the form of the model functions. The challenge for approximating nonlinear model functions is, that based on the theory in \cite{düll_gwiazda_marciniak-czochra_skrzeczkowski_2021}, the approximation needs to Lipschitz continuous in the nonlinearity with Lipschitz constants that are uniformly bounded as we decrease the mesh size, which is needed for controlling the overall error estimate. Unfortunately, this cannot be obtained with the approximation technique that we propose, see Remark \ref{Remark:issuesNonLinear}. 

In the setting of this paper, the transport function is described by a push forward and does hence not need derivatives or velocity fields. The trajectories of the transport function might be much easier to estimate from given data compared to estimating derivatives. 
Overall, the model framework tat we based our analysis on, allows analytical and now also numerical analysis of a broad spectrum of relatively low regularity models, which is promising for many different applications such as cell dynamics, coagulation-fragmentation models and crowd dynamics.

In the future, we will investigate the how the algorithm can be improved for particular models. We want to study how the assumptions can be relaxed or how the accuracy of the algorithm can be increased. For example if we can discretize the space at a low cost and thus adaptively choose the grid.\\

Based on Section \ref{Bayes}, we want to look into practical numerical methods to sample from the posterior measure. This will allow us to estimate the parameters determining the dynamics of a solution to our structured population model and to accurately model the dynamics describing real experimental data.\\

\newpage

\nomenclature[M]{$(U,d)$}{State space, assumed to be separable and complete}
\nomenclature[M]{$\mu(t)$}{Solution to the structured population model \eqref{eq:mod_pushforward} at time $t$}
\nomenclature[M]{$b$}{Transport direction and velocity of the structured population model on Euclidean spaces}
\nomenclature[M]{$c$}{Real valued model function describing growth and decay}
\nomenclature[M]{$\eta$}{Measure-valued model function describing non-local transitions in the state space}
\nomenclature[M]{$N$}{Measure-valued model function describing influx}
\nomenclature[M]{$X$}{Model function describing the transport through the state space}

%Numercial scheme
\nomenclature[N]{$r$}{Bound for the tails of the solution $\mu(\cdot)$}
\nomenclature[N]{$K$}{Compact set on which we approximate the solution $\mu(\cdot)$}
 \nomenclature[N]{$\varepsilon$}{Distance between the points used to cover $K$}
 \nomenclature[N]{$(B_\varepsilon(z_l))_{l=1,...,L}$}{Finite, $\varepsilon$ dense covering of $K$}
\nomenclature[N]{$(z_l)_{l=1,...,L}$}{Center points of the finite covering of $K$ }

 \nomenclature[N]{$\Delta t$}{Size of time step}
 \nomenclature[N]{$\alpha$}{Hölder regularity of the model functions  in time}
\nomenclature[N]{$\mu_k$}{Numerical approximation of $\mu(k\Delta t)$ (see Section \ref{Sec:NumericalScheme})}
\nomenclature[N]{$\beta_l$}{Value of the model function $\hat\eta$ in the grid point $z_l$}
\nomenclature[N]{$\sigma_l$}{Value of the model function $\hat N$ in the grid point $z_l$}
\nomenclature[N]{$x_j^{k}$}{Locations of the $J_k$ Dirac masses in the numerical approximation at time $k\Delta t$}
\nomenclature[N]{$m_j^{k}$}{Masses of the $J_k$ Dirac masses in the numerical approximation at time $k\Delta t$}

\nomenclature[Z]{$C^0_b(U)$}{Continuous and bounded functions}

\nomenclature[Z]{$BL(U)$}{Bounded and Lipchitz continuous functions}

%Bayesian inverse
 \nomenclature[P]{$\pi$}{Prior Distribution}
 \nomenclature[P]{$\ell$}{Likelihood function, see \eqref{Eq:likelyhood}}
 \nomenclature[P]{$\mathcal{G}$}{Solution operator mapping parameters to a solution, see \eqref{Eq:solutionOperator}}
\nomenclature[P]{$Y\in\mathbb{R}^{I\cdot M}$}{Data vector describing data at $M$ time points in $I$ "directions" $g_i\in BL(U)$}
\nomenclature[P]{$\gamma$}{Hölder regularity of the noise density}

%finite range approximation
\nomenclature[O]{$\mathcal{K}$}{$=\bigcup_{l=1}^L B_{\varepsilon}(z_l)$, the  support of the partition of unity used for the approximation}
\nomenclature[O]{$(\varphi_l)_{l=1,...,L}$}{Partition of unity subordinate to $(B_\varepsilon(z_l))_{l=1,...,L}$}
\nomenclature[O]{$\hat\eta$}{Finite-Range approximation of model function $\eta$, see Theorem \ref{Theorem:finiteRangeApproximation}}
\nomenclature[O]{$\hat N$}{Finite-Range approximation of model function $N$, see Theorem \ref{Theorem:finiteRangeApproximation}}

%Semigroup Splitting
\nomenclature[Z]{$\mathcal{S}^1$}{Semigroup describing transport}
\nomenclature[Z]{$\mathcal{S}^2$}{Semigroup describing growth, non-local transitions and influx}
\nomenclature[Z]{$K_1,K_2,C_1,C_2$}{Constants associated with the exponential Lipschitz continuity of the semigroups}

%Conclusion
\nomenclature[Z]{$d$}{Dimension of Euclidean space}
\nomenclature[Z]{$\omega_{X,R}(t)$}{Modulus of continuity for the regularity of $X$ w.r.t. time}
\nomenclature[Z]{$L_{X,R}(t)$}{Function describing the continuity of $X$ w.r.t. space}
\nomenclature[Z]{$\delta_x(\cdot)$}{Dirac measure centered in $x$}

\renewcommand{\nomname}{List of Symbols}

\printnomenclature[6em]
\newpage
\bibliographystyle{plainnat}

\bibliography{bibfile}

\newpage
\appendix

\section{Assumptions on model functions}
\label{Appendix:Assumptions}
In this appendix, we give an overview of the assumptions on the model functions that are needed for the well-posedness of the model and the convergence of the numerical scheme.\\
Recall that we use the following model functions for the structured population model \eqref{eq:mod_pushforward}:
\begin{align}
\begin{split}
&c: [0,T]\times U\times\mathcal{M}^+(U)\rightarrow \mathbb{R}\\
&N:[0,T]\times\mathcal{M}^+(U)\rightarrow \mathcal{M}^+(U)\\
& X:[0,T]\times[0,T]\times U\times([0,T]\rightarrow\mathcal{M}^+(U))\rightarrow U\\
&\eta:[0,T]\times U\times\mathcal{M}^+(U)\rightarrow\mathcal{M}^+(U).
\end{split}
\end{align}
Note that the functions $N$ and $\eta$ take values in $\mathcal{M}^+(U)$.

We make the following assumptions on the model functions:
\begin{assumptions}
\label{assumptionsOnModelFunctions}
       
\begin{enumerate}
    \item  The functions $\eta$ and $N$ are uniformly tight:
for $r>0$ there exists a compact set K such that 
\begin{align}
    \vert \eta(t,x,\mu)(U\backslash K)\vert\leq r 
    \text{ and }
    \vert N(t,\mu)(U\backslash K)\vert\leq r 
\end{align}
for all $t\in[0,T]$, $\mu\in\mathcal{M}^+(U)$ and $x\in U$.
\label{Ass:ModelFunctionsTight}

\item We assume that $\eta$ is uniformly narrowly continuous in $x$
    \begin{align}
    \sup_{t\in[0,T]}\sup_{\mu\in\mathcal{M}^+(U)}\Vert\eta(t,\cdot,\mu)\Vert_{C^0(\mathcal{M}^+)}<\infty
    \end{align}
and that
       \begin{align}
           \sup_{t\in[0,T]}\sup_{\mu\in\mathcal{M}^+(U)}\Vert N(t,\mu)\Vert_{BL^*}<\infty,
       \end{align}
    as well as

\begin{equation}
     \sup_{t\in[0,T]} \sup_{\mu\in\mathcal{M}^+(U)} \Vert c(t,\cdot,\mu) \Vert_{BL} <\infty.   
\end{equation}
\label{assumptionsSupremum}

\item For all $R>0$ there exist constants $L_{R,c},L_{R,\eta}, L_{R,N}$ such that for any measures $\mu,\nu\in\mathcal{M}^+(U)$ with $\Vert\mu\Vert_{BL^*},\Vert\nu\Vert_{BL^*}\leq R$:
\begin{align}
&\sup_{t\in[0,T]}\sup_{x\in U}\vert c(t,\cdot,\mu)-c(t,\cdot,\nu)\vert \leq L_{R,c}\ \rho_F(\mu,\nu), \\
&\sup_{t\in[0,T]}\sup\limits_{x\in U}\Vert \eta(t,x,\mu)-\eta(t,x,\nu)\Vert_{BL^*}
\leq L_{R,\eta}\ \rho_F(\mu,\nu) \text{ and } \\
&\sup_{t\in[0,T]}\Vert N(t,\mu)-N(t,\nu)\Vert_{BL^*} \leq L_{R,N}\ \rho_F(\mu,\nu) .
\end{align}
\label{assumptionsLipschitzMeasure}

\item  For $\alpha\in(0,1]$, we assume that all of the model functions are uniformly of $C^{0,\alpha}$ regularity w.r.t. time: for all $\mu\in\mathcal{M}^+(U)$, $x\in U$ and $t,\tau\in[0,T]$:
       \begin{align*}
       & \vert c(t,x,\mu)-c(\tau,x,\mu)\vert \leq \Vert c\Vert_{0,\alpha}\vert t-\tau\vert ^\alpha,\\
           & \Vert \eta(t,x,\mu)-\eta(\tau,x,\mu)\Vert_{BL^*}
           \leq \Vert \eta\Vert_{0,\alpha} \vert t-\tau\vert^\alpha \text{ and }\\
           & \Vert N(t,\mu)-N(\tau,\mu)\Vert_{BL^*} \leq \Vert N\Vert_{0,\alpha} \vert t-\tau\vert^\alpha,
       \end{align*}
       where we assume that all of the constants are uniformly bounded with respect to space and measure variables.
 \label{assumptionsOnModelFunctionsTime}
       
\item We assume that for fixed $x\in U$ and $\mu\in\mathcal{M}^+(U)$ 
       \begin{align}
           c(\cdot,x,\mu)\in C^1([0,T],\mathbb{R}).
       \end{align}
       In the case that $\eta$ and $N$ are linear we assume that for test functions $\psi\in BL(U)$
       \begin{align}
       &\int\limits_U \psi(y)\mathrm d [\eta(t,x)](y) \in C^1([0,T])\text{ and }\\
       &\int\limits_U \psi(y)\mathrm d [N(t)](y) \in C^1([0,T])\\
       \end{align}
       for all $x\in U,\mu\in\mathcal{M}^+(U)$.
       \label{C1assumption}

\item Let $R>0$. For a narrowly continuous map 
$\mu(\cdot) :[0,T]\rightarrow \mathcal{M}^+(U)$ such that $\sup\limits_{t\in[0,T]} \Vert\mu(t)\Vert_{BL^*}\leq R, $
the function $X(t,\tau,x,\mu(\cdot))$ is uniformly continuous in $t$ and Lipschitz continuous in $x$:\\ 
 We assume that there exists a modulus of continuity ${\omega_{X,R}:[0,\infty]\rightarrow [0,\infty]}$ and a function $L_{X;R}:\mathbb{R}\rightarrow\mathbb{R}$ locally bounded such that for all $t_1,t_2\in[0,T]$ and $x_1,x_2\in U$
\begin{align*}
&\sup\limits_{\tau\in[0,T]}\sup\limits_{x\in U} d\big(X(t_1,\tau,x,\mu(\cdot)),X(t_2,\tau,x,\mu(\cdot))\big)
\leq \omega_{X,R}\big(\vert t_1-t_2\vert\big)\\
\intertext{and}
&d\big(X(t_2,t_1,x_1,\mu(\cdot)),X(t_2,t_1,x_2,\mu(\cdot))\big)\leq L_{X,R}(t_2-t_1)d\big(x_1,x_2\big).
\end{align*}
It is further assumed that $\omega_{X,r}$ and $L_{X,R}$ decay as follows:
\begin{align}
   \omega_{X,R}(t)\leq Ct \text{ and } L_{X,R}(t)\leq e^{Ct}.
\end{align}

\label{assumptionsXLipschitzTX}

\item 
For all $ R>0$ there exist functions $ L_{R,X}\in L^1([0,T])$ such that for all families of measures $\mu(\cdot),\ \nu(\cdot)\in C^0([0,T],\mathcal{M}^+(U))$ with $
{\sup\limits_{t\in[0,T]}\Vert\mu(t)\Vert_{BL^*}\leq R} $ and $ {\sup\limits_{t\in[0,T]}\Vert\nu(t)\Vert_{BL^*}\leq R }$
it holds for $t_1,t_2\in[0,T]$ that 
\begin{align}
     \sup_{x \in U} d(X(t_2,t_1,x,\mu(\cdot)),X(t_2,t_1,x,\nu(\cdot)) )\leq\int\limits_{t_1}^{t_2} L_{R,X}(\tau) \rho_F\big(\mu(\tau),\nu(\tau)\big)\mathrm{d}\tau .
\end{align}
\label{assumptionsXLipschitzMeasure}

\item 
For a narrowly continuous map $\mu(\cdot) :[0,T]\rightarrow \mathcal{M}^+(U)$, the function $X(t,\tau,x,\mu(\cdot))$ has the semigroup property:
 \begin{align}
     X(\tau_2,\tau_1,\cdot,\mu(\cdot))=X(\tau_2,t,\cdot,\mu(\cdot))\circ X(t,\tau_1,\cdot,\mu(\cdot)).
 \end{align}

\label{AssumptionSemigroup}

\item For nonlinear $\eta$ and $N$, we need to assume that they are of the following form
\begin{align}
    \eta(t,x,\mu)(\cdot)&=\sum_{l=1}^L \beta_l(t,x,\mu)\delta_{z_l}(\cdot)\\
    N(t,\mu)(\cdot)&=\sum_{l=1}^L n_l(t,x,\mu)\delta_{z_l}(\cdot),
\end{align}
where $\{z_l\vert l=1,..,L\}$ are the grid points, see Section \ref{Sec:NumericalScheme}. We impose that the functions $\beta_l(\cdot,x,\mu)$, $N(\cdot,\mu)\in C^1([0,T])\cap C^{0,\alpha}([0,T])$ for any fixed $x\in U$ and $\mu\in\mathcal{M}^+(U)$. We further impose that the functions are uniformly Lipschitz continuous in the measure argument, i.e. for any $R>0$ and $\mu,\nu\in \mathcal M^+(U)$ such that $\Vert \mu\Vert_{BL^*},\Vert\nu\Vert_{BL^*}\leq R$
\begin{align}
    &\sup_{t\in[0,T]}\sup_{x\in U}\vert \beta_l(t,\cdot,\mu)-\beta_l(t,\cdot,\nu)\vert \leq L_{R,\eta}\ \rho_F(\mu,\nu) \text{ and } \\
&\sup_{t\in[0,T]}\Vert n_l(t,\mu)-n_l(t,\nu)\Vert_{BL^*} \leq L_{R,N}\ \rho_F(\mu,\nu) 
\end{align}
for $l=1,...,L$.\\
Lastly, we need to assume that $\beta_l(t,\cdot,\mu)$ is continuous in $x$ for $t$ and $\mu$ fixed. 
\label{Ass:finiteRangeNonlinear}
\end{enumerate}

\end{assumptions}

\begin{remark}
    It is easy to verify that for the approximations $\hat\eta$ and $\hat N$ from Theorem \ref{Theorem:finiteRangeApproximation}, Assumptions \ref{assumptionsOnModelFunctions}\eqref{assumptionsSupremum} and  \eqref{assumptionsOnModelFunctionsTime} are also fulfilled. Further, $\hat\eta$ and $\hat N$ are only supported on the  set $\mathcal{K}=\bigcup_{l=1}^LB_\varepsilon(z_l)$ and hence tight. 
\end{remark}

\begin{remark}
\label{AssumptionsX}

    In \cite[Assumptions 3.12]{düll_gwiazda_marciniak-czochra_skrzeczkowski_2021}, the much stronger condition that $X(t,\tau,\cdot,\mu(\cdot))$ is a bi-Lipschitz family of homeomorphisms is required. This is satisfied for solutions of flows on vector field on $\mathbb{R}^d$ of the form 
    \begin{align}
        \partial_t X_b(t,\tau,x,\mu(\cdot))=b(t,X_b(t,\tau,x,\mu(\cdot)),\mu(t)),\hspace{2 cm} X_b(t,\tau,x,\mu(\cdot))=x
    \end{align} by \cite[Lemma 3.29]{düll_gwiazda_marciniak-czochra_skrzeczkowski_2021} 
    under the assumption that $b(\cdot,\cdot,\mu)\in L^1([0,T],BL(\mathbb{R}^d))$ uniformly in $\mu$ and that $b$ is Lipschitz continuous in the measure argument. Such flows of vector fields are the motivation to consider the transport functions $X$ in the model \eqref{eq:mod_pushforward}. However, a careful study of the arguments in \cite[Chapter 3]{düll_gwiazda_marciniak-czochra_skrzeczkowski_2021} reveals that it is sufficient to consider a function $X$ which is Lipschitz but not necessarily bijective (see Assumption \ref{assumptionsOnModelFunctions}\eqref{assumptionsXLipschitzTX})\cite{Christian}. This allows to model many more phenomena.\\
    In the formulation of $X$ as the flow of a vector field, it becomes visible why $X$ depends on $\mu(\cdot)$ at all time points as the location at time $t_1$ might depend on the solution at all times between $t_0$ and $t_1$. 
\end{remark}

\section{The space \texorpdfstring{$\mathcal{M}^+(U)$} and solutions of structured population models on the space of Radon measures}
\label{Sec:SolutionsandProperties}
We recapitulate the results from \cite{düll_gwiazda_marciniak-czochra_skrzeczkowski_2021} on the space $(\mathcal M^+(U),\rho_F)$ and its properties that are important for this work. 

\begin{lemma}[{Scaling of the flat norm \cite[1.23]{düll_gwiazda_marciniak-czochra_skrzeczkowski_2021}}]
\label{scalingFlatNorm}
    For $s>0$, the flat norm has the following scaling property:
    \begin{align}
        \Vert \mu\Vert_{BL^*}=\frac{1}{s}\sup \left\{\int\limits_U \psi(x)\mathrm{d}\mu(x),\ \vert \psi\in BL(U),\ \Vert \psi\Vert_{BL}\leq s\right\}.
    \end{align}
\end{lemma}
\begin{theorem}[Properties of $(\mathcal{M}^+(U),\rho_F)$]
\label{propertiesM+}
If $U$ is Polish, then 
    the space $(\mathcal{M}^+(U),\rho_F)$ is separable \cite[Theorem 1.37]{düll_gwiazda_marciniak-czochra_skrzeczkowski_2021} and complete \cite[Theorem G.42]{düll_gwiazda_marciniak-czochra_skrzeczkowski_2021}.\\
    Further, for $\mu\in\mathcal{M}^+(U)$, it holds that the total variation norm and the flat norm coincide
    \begin{align*}
        \Vert\mu\Vert_{TV}:=\mu(U)=\Vert \mu\Vert_{BL^*},
    \end{align*}
    see \cite[Proposition 1.44]{düll_gwiazda_marciniak-czochra_skrzeczkowski_2021}.
\end{theorem}

In the following theorem, we summarize some results on the existence, continuity and tightness of the solutions to \eqref{eq:mod_pushforward}.

\begin{theorem}{\cite[Chapter 3]{düll_gwiazda_marciniak-czochra_skrzeczkowski_2021}}\\
    \noindent 
    Under the assumptions given in \ref{assumptionsOnModelFunctions}, the model \eqref{eq:mod_pushforward} is well posed. For any given initial measure $\mu(0)\in\mathcal{M}^+(U)$, a unique solution $(\mu(t))_{t\in[0,T]}$ exists on $[0,T]$.\\
    \begin{enumerate}
        \item  Its total variation norm is uniformly bounded, i.e. there exists a constant $C$ such that
    
    \begin{equation}
    \label{uniformBoundTV}
        \sup_{t\in[0,T]}\Vert \mu(t)\Vert_{TV}\leq C.
    \end{equation}

    \item  The solution is Lipschitz continuous w.r.t. time:
    \begin{equation}
        \label{Lipschitz_time}
        \rho_F\big(\mu({t_1}),\mu({t_2})\big)
        \leq C\vert t_1-t_2\vert \text{ for all } t_1,t_2\in[0,T].
    \end{equation}

    \item    
    The solution is tight, i.e. for every $\varepsilon>0$, there exists a compact set $K\subset U$, such that
    \begin{align}
        \mu(t)(U\backslash K)\leq \varepsilon
    \end{align}
    for all $t\in[0,T]$.\\

    \item  The model is Lipschitz continuous with respect to initial values:\\    
    Let $(\mu(t))_{t\in[0,T]},(\nu(t))_{t\in[0,T]}$ be solutions to \eqref{eq:mod_pushforward} with initial values $\mu(0)$ and $\nu(0)$ respectively. There exists a constant $C$ s.t.
    \begin{align}
        \label{Lipschitz_initial}
        \rho_F\big(\mu(t),\nu(t)\big)
        \leq C \rho_F\big(\mu(0),\nu(0)\big)e^{\int\limits_0^t \sup\limits_{\nu\in\mathcal{M}^+(U)}\Vert \eta(\tau,\cdot,\nu)\Vert_{BL^*}\mathrm{d}\tau}.
    \end{align}

    \item  
    For two solutions $\mu_1(\cdot),\ \mu_2(\cdot)$ with the same initial value but with different model functions $X_1,c_1,N_1,\eta_1$ and  $X_2,c_2,N_2,\eta_2$ fulfilling Assumptions \ref{assumptionsOnModelFunctions}, it holds that
    \begin{align}
        \begin{split}
          \label{Lipschitz_model_functions}
        &\rho_F\big(\mu_1(t),\mu_2(t)\big)\leq  C\int\limits_0^t \Vert c_1(\tau,\cdot,\mu_1(\tau))-c_2(\tau,\cdot,\mu_2(\tau))\Vert_\infty \mathrm{d}\tau\\
        &\quad+ C \int\limits_0^t \sup_{x\in U} \rho_F\big(\eta_1(\tau,x,\mu_1(\tau)),\eta_2(\tau,x,\mu_2(\tau))\big)+\rho_F\big(N_1(\tau,\mu_1(\tau)), N_2(\tau,\mu_2(\tau))\big)\mathrm{d}\tau\\
        &\quad+C \sup_{0\leq \tau_1\leq\tau_2\leq t} \sup_{x\in U} d\big(X_1(\tau_2,\tau_1,\cdot,\mu_1(\cdot)),X_2(\tau_2,\tau_1,\cdot, \mu_2(\cdot))\big) .
        \end{split}
    \end{align}
    \item The solutions to \eqref{eq:mod_pushforward} generate a two-parameter semigroup.
    
    \end{enumerate} 
    \label{Thm:solution}
\end{theorem}

\begin{proof}

    The existence of the solutions as well as the Lipschitz continuity w.r.t. the initial values now follow directly from \cite[Theorem 3.31]{düll_gwiazda_marciniak-czochra_skrzeczkowski_2021}. Estimates \eqref{uniformBoundTV}, \eqref{Lipschitz_time} and \eqref{Lipschitz_model_functions} follow from \cite[Lemma 3.23, Lemma 3.26 and Lemma 3.21]{düll_gwiazda_marciniak-czochra_skrzeczkowski_2021} using the fact that solutions of the nonlinear model \eqref{eq:mod_pushforward} arise from a fixed point argument on the corresponding linear model. A careful study of the proofs in \cite[Chapter 3.2]{düll_gwiazda_marciniak-czochra_skrzeczkowski_2021} reveals that narrow continuity of $\eta$ in $x$ with uniformly bounded norm $\Vert\eta(t,x,\mu)\Vert_{BL^*}$ is sufficient for the existence and uniqueness of solutions. While the function $x\mapsto \int\limits_U \psi(y)\mathrm{d} [\eta(t,x)](y)$ for fixed $\psi\in BL(U)$ is no longer Lipschitz continuous (compare \cite[Lemma 3.20]{düll_gwiazda_marciniak-czochra_skrzeczkowski_2021}), we can avoid this by using $\psi(x)=1\in BL(U)$ as a test functions in the proofs of \cite[Theorem 3.22 and Theorem 3.24]{düll_gwiazda_marciniak-czochra_skrzeczkowski_2021} instead.\\
    The tightness of the solutions follows from the narrow continuity of the solutions by Prokhorov's theorem. For every subsequence $(\mu(t_n))_{n\in\mathbb{N}}$ of $\{\mu(t)\vert t\in[0,T]\}$, the sequence $(t_n)_{n\in\mathbb{N}}\subset[0,T]$ has a convergent subsequence with limit $t\in[0,T]$ by Heine-Borel's theorem. As the solution $\mu(\cdot)$ is narrowly continuous, $t_{n_k}\rightarrow t$ as $k\rightarrow \infty$ implies $\mu({t_{n_k}})\rightarrow \mu(t)$ narrowly. By Prokhorov's theorem \cite{Prokhorov}, the family of measures $\{\mu(t)\vert t\in[0,T]\}$ is tight.
\end{proof}

\begin{remark}
    The tightness of the solutions to \eqref{eq:mod_pushforward} justifies the approximation of the solutions on a subset of the space $U$ in Section \ref{Sec:NumericalScheme}. By approximating the solution on a sufficiently large subset, we capture the essential behavior of the solutions.
\end{remark}

\begin{lemma}[Weak formulation]
\label{Lemmma:weakform}
The structured population model \eqref{eq:mod_pushforward} admits the following weak formulation: a narrowly continuous family of measures 
$(\mu(t))_{t\in[0,T]}$ is a solution in the sense of Definition \ref{structuredPopulationModel} if and only if for all test functions $\psi\in BL(U)$the following is true
\begin{align}
\label{eq:Weakform}
    &\int\limits_U\psi(x) \mathrm{d}[\mu(t)] (x)
= \int\limits_U \psi\big(X(t,0,x,\mu(\cdot))\big)e^{\int\limits_0^t c(s,X(s,0,x,\mu(\cdot)),\mu_0)\mathrm{d}s}\mathrm{d}[\mu(0)](x)\\
&\quad+ \int\limits_0^t\int\limits_U\int\limits_U
\psi\big(X(t,\tau,y,\mu(\cdot))\big)e^{\int\limits_\tau^t c(s,X(s,\tau,y,\mu(\cdot)),\mu(s))\mathrm{d}s}\mathrm{d}[\eta(\tau,x,\mu(\tau))](y) \mathrm{d}[\mu(\tau)](x) \mathrm{d}\tau\\
&\quad+\int\limits_0^t\int\limits_U
\psi\big(X(t,\tau,y,\mu(\cdot))\big)e^{\int\limits_\tau^t c(s,X(s,\tau,y,\mu(\cdot)),\mu(s))\mathrm{d}s}\mathrm{d}[N(\tau,\mu(\tau))](y) \mathrm{d}\tau.
\end{align}
    
\end{lemma}

\begin{proof}
    If $(\mu(t))_{t\in[0,T]}$ is a family of solutions, the formulation \eqref{eq:Weakform} follows directly form \cite[Lemma 3.15]{düll_gwiazda_marciniak-czochra_skrzeczkowski_2021} using the fact that the nonlinear solutions of the nonlinear model are fixed points of the solution operator of the linear model.\\
    The reverse direction follows by applying Hahn-Banach Theorem to the space $\overline{\mathcal M (U)}^{\Vert\cdot\Vert_{BL^*}}$ which has $BL(U)$ as its dual space \cite[Theorem 1.41]{düll_gwiazda_marciniak-czochra_skrzeczkowski_2021}.
\end{proof}

\section{Additional Lemmata}
In this section, we state some useful lemmata that are used throughout the paper.
\begin{lemma}{\cite[Lemma 2.22]{düll_gwiazda_marciniak-czochra_skrzeczkowski_2021}}
\label{iterationLemma}
    Let $a,b,\in\mathbb R_{\geq 0}$. If a sequence $(u_k)_{k\in\mathbb{N}}$ satisfies 
    \begin{align}
        \vert u_{k+1}\vert
        \leq a \vert u_k\vert +b
    \end{align}
    for $k\in\mathbb{N}$, then it holds that
    \begin{align}
        \vert u_k\vert \leq a^k\vert u_0\vert + 
        \begin{cases}
            \frac{a^k-1}{a-1}b &\text{ if }a\neq 1\\
            kb &\text{ else }
        \end{cases}
    \end{align}
\end{lemma}

The following Lemma is easily verified by a short calculation:
\begin{lemma}
\label{boundednessZeta}
Given a measure $\nu\in\mathcal{M}^+(U)$ and times $t,\tau\in[0,T]$, the function $x\mapsto\exp\left(\int\limits_\tau^t c(s,x,\nu)\mathrm{d}s\right)$ is a bounded Lipschitz function.
\end{lemma}

\end{document}